\documentclass[12pt,a4paper,reqno]{amsart} 
\usepackage[utf8]{inputenc} 
\usepackage[top=3cm, bottom=3.5cm, left=2.5cm, right=2.5cm]{geometry}
\usepackage{amsmath,amsthm,amsfonts,amstext,amssymb}
\usepackage{mathrsfs}
\usepackage{enumerate}
\usepackage{hyperref}
\usepackage{dsfont}
\usepackage{color}
\usepackage{centernot}
\usepackage{graphicx}
\usepackage{xcolor}
\hypersetup{
    colorlinks,
    linkcolor={blue!50!black},
    citecolor={blue!50!black},
    urlcolor={blue!80!black}
}

\usepackage{etoolbox}
\patchcmd{\section}{\scshape}{\bfseries\large}{}{}
\patchcmd{\subsubsection}{\itshape}{}{}{}
\makeatletter
\def\@seccntformat#1{\csname the#1\endcsname.\space}
\makeatother

\newtheorem{theorem}{Theorem}[section]

\newtheorem{lemma}[theorem]{Lemma}
\newtheorem{proposition}[theorem]{Proposition}

\theoremstyle{definition}

\newtheorem{remark}[theorem]{Remark}

\numberwithin{equation}{section}

\usepackage{booktabs}
\usepackage{mathtools}
\newcommand{\C}{\mathbb{C}}
\newcommand{\E}{\mathbb{E}}
\newcommand{\N}{\mathbb{N}}
\renewcommand{\P}{\mathbb{P}}

\newcommand{\R}{\mathbb{R}}
\newcommand{\Z}{\mathbb{Z}}

\newcommand{\cC}{{\ensuremath{\mathcal C}} }

\begin{document}

\title{Critical curve of loop percolation on the $d$-regular tree}

\author[L. Makowiec]{Luca Makowiec}
\address{University of Leipzig\\
Department of Mathematics\\
Augustusplatz 10, 04109 Leipzig, Germany.}
\email{luca.makowiec@uni-leipzig.de}

\thanks{Email: \href{mailto:luca.makowiec@uni-leipzig.de}{\texttt{luca.makowiec@uni-leipzig.de}},
\href{mailto:artem.sapozhnikov@math.uni-leipzig.de}{\texttt{artem.sapozhnikov@math.uni-leipzig.de}}} 
\author[A. Sapozhnikov]{Artem Sapozhnikov}
\address{University of Leipzig\\
Department of Mathematics\\
Augustusplatz 10, 04109 Leipzig, Germany.}
\email{artem.sapozhnikov@math.uni-leipzig.de}


\begin{abstract}
We consider clusters formed by a Poisson ensemble of random walk loops on the $d$-regular tree with an intensity parameter $\alpha>0$ and a killing parameter $\kappa>-1$; the latter penalizes ($\kappa > 0$) or favors ($\kappa <0$) the appearance of large loops. 
We obtain an implicit formula for the critical curve $\kappa\mapsto \alpha_c(\kappa)$ for the percolation phase transition; the curve is positive if and only if $\kappa>\kappa_c = \frac{2\sqrt{d-1}}{d}-1$, differentiable away from $\kappa_c$, and has order $\sqrt{\kappa-\kappa_c}$ as $\kappa\downarrow\kappa_c$ and order $(1+\kappa)^2$ as $\kappa\to\infty$. We show that for each $\kappa>-1$, an infinite cluster exists exactly when $\alpha>\alpha_c(\kappa)$.
Finally, we identify the near-critical behavior of the susceptibility and the percolation probability: for $\kappa>\kappa_c$, the critical exponents take the mean-field values, while for $\kappa=\kappa_c$, the phase transition is of a higher order with the percolation probability decaying quadratically in $\alpha-\alpha_c$.

    \end{abstract}

\maketitle

\section{Introduction}
Let $\mathbb{T}_d$ be the $d$-regular infinite tree.
A (non-trivial discrete) \emph{based loop} on $\mathbb{T}_d$ is a nearest neighbor path $\Dot{\ell} = (v_1, v_2, \ldots, v_n)$ on $\mathbb{T}_d$ with $v_n$ being a neighbor of $v_1$. 
The equivalence classes of based loops modulo cyclic shifts of their vertices are called \emph{loops}. 
For $\kappa > -1$, consider the measure $\Dot{\mu}_\kappa(\cdot)$ on the space of based loops that gives weight 
\[
\Dot{\mu}_\kappa(\Dot{\ell}) = \frac{1}{n}\Big(\frac{1}{1+\kappa}\Big)^n \Big(\frac{1}{d}\Big)^n
\]
to each based loop $\Dot{\ell} = (v_1, v_2, \ldots, v_n)$, and denote by $\mu_\kappa$ the pushforward of $\Dot{\mu}_\kappa$ on the space of loops. 
We are interested in the Poisson ensembles of loops $\mathcal{L}_{\alpha, \kappa}$ with intensity $\alpha \mu_\kappa$, where 
$\alpha\geq 0$ governs the amount of loops entering the picture and $\kappa>-1$ plays the role of killing on vertices penalizing ($\kappa>0$) or favoring ($\kappa<0$) appearance of large loops. 
(The special case of $\kappa=0$ is known as the random walk loop soup.)
We write $\P_{\alpha,\kappa}$ for the law of $\mathcal{L}_{\alpha,\kappa}$.

\smallskip

In this paper, we study percolation of loops in $\mathcal{L}_{\alpha,\kappa}$. 
We say that a loop $\ell$ is \emph{open} if it is present at least once in $\mathcal{L}_{\alpha,\kappa}$.
Two vertices $u$ and $v$ are loop-connected if there exists a sequence of open loops $(\ell_1,\ldots,\ell_k)$ such that $u$ is a vertex of $\ell_1$, $v$ is a vertex of $\ell_k$, and each pair of loops $\ell_{i-1}, \ell_i$ of the sequence shares a vertex. 
We write $\mathcal C(u)$ for the open cluster of $u$, that is, the set of all vertices loop-connected to $u$, and consider the percolation probability 
\[
\theta(\alpha, \kappa) = \P_{\alpha,\kappa}\big(|\mathcal C(0)| = \infty\big),
\]
where $0$ is a fixed vertex of $\mathbb{T}_d$ (called the root or the origin). 
We are interested in the \emph{critical curve} for the percolation phase transition
\[
\alpha_c(\kappa) = \inf\{ \alpha \geq 0 : \theta(\alpha,\kappa) > 0\},\quad \kappa>-1.
\]
From the monotonicity of the intensity measure of $\mathcal{L}_{\alpha,\kappa}$ in $\alpha$ and $\kappa$, it follows that $\theta(\alpha,\kappa)$ is a non-decreasing function of $\alpha$ and a non-increasing function of $\kappa$. 
In particular, $\alpha_c(\kappa)$ is non-decreasing in $\kappa$, and for each $\kappa$, $\theta(\alpha,\kappa)>0$ if $\alpha>\alpha_c(\kappa)$ and $\theta(\alpha,\kappa)=0$ if $\alpha<\alpha_c(\kappa)$. 

\smallskip

The loop percolation induced by $\mathcal{L}_{\alpha,\kappa}$ dominates stochastically the Bernoulli bond percolation on $\mathbb{T}_d$ induced by the restriction of $\mathcal{L}_{\alpha,\kappa}$ to the loops of length $2$, from which it follows that $\alpha_c(\kappa)<\infty$ for all $\kappa$, see \cite[Proposition~4.3]{LJL13}. 
We emphasize however that for $\kappa\leq 0$, the loop percolation behaves differently from the Bernoulli percolation due to its long-range correlations, see e.g.\ Remark~\ref{rem:longrange}. In particular, verifying if $\alpha_c(\kappa)$ is strictly positive becomes a non-trivial task; see \cite{CS16} for the proof that $\alpha_c(\kappa)>0$ iff $\kappa\geq 0$ for the loop percolation on $\mathbb Z^d$. 

\smallskip

\subsection{Results}
Our main result gives an implicit characterization of $\alpha_c(\kappa)$ as a unique solution to some equation. Let 
\begin{equation}\label{def:kappa-c}
\kappa_c = \frac{2\sqrt{d-1}}{d} - 1
\end{equation}
and for $\kappa\geq \kappa_c$, define
\begin{equation}\label{def:tau}
\tau_\kappa=  \frac{d (1+\kappa) - \sqrt{d^2 (1+\kappa)^2 - 4(d-1)}}{2(d-1)}.
\end{equation}
Note that $\tau_\kappa \in \big(0, \frac{1}{\sqrt{d-1}}\big]$, $\tau_{\kappa_c} = \frac{1}{\sqrt{d-1}}$ and $\tau_\kappa\downarrow 0$ as $\kappa\uparrow\infty$. 

\begin{theorem} \label{thm:alpha-c-kappa}
The critical value $\alpha_c(\kappa)$ is positive if and only if $\kappa>\kappa_c$. Furthermore, for each $\kappa>\kappa_c$, 
$\alpha_c(\kappa)$ is the unique solution of $\alpha > 0$ to the equation \begin{equation}\label{eq:alpha-c-kappa}
\sum_{n \geq 0} \Big( \big(1-\tau_\kappa^{2(n+1)}\big)^{-\alpha} -1 \Big) (d-1)^n  = \frac{1}{d-2}.
\end{equation}
\end{theorem}
\begin{remark}
Note that $2\sqrt{d-1}/d$, appearing in \eqref{def:kappa-c}, is the spectral radius of the random walk on $\mathbb{T}_d$, $\rho = \lim_{n\to \infty} p_{2n}(0,0)^{\frac{1}{2n}}$, where $p_n(\cdot,\cdot)$ is the $n$-step transition probability of the random walk, see the proof of Lemma~\ref{l:greens_f}. 
\end{remark}
Although the solution to \eqref{eq:alpha-c-kappa} is not explicit, the formula is sufficient to establish basic regularity of $\alpha_c(\kappa)$ and derive its asymptotic expansions as $\kappa \downarrow \kappa_c$ and $\kappa \to \infty$.
\begin{figure}[t]
    \centering
    \includegraphics[width=0.7\textwidth]{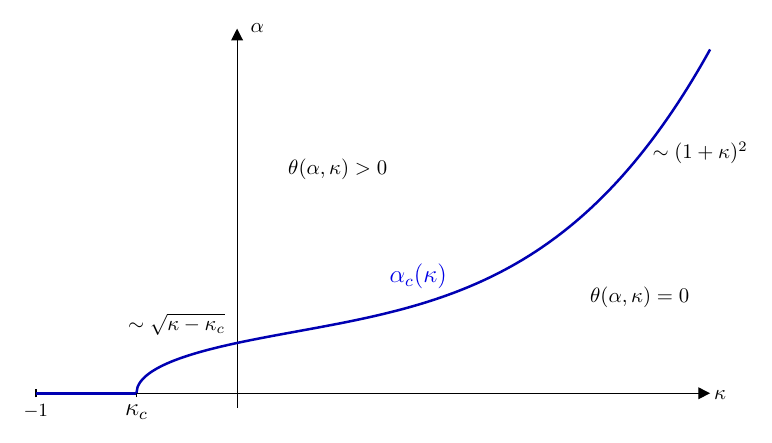}
    \caption{The critical curve $\alpha_c(\kappa)$. Above the curve there is an infinite cluster, while on and below the curve there is no infinite cluster.}
    \label{fig:crit_curve}
\end{figure}

\begin{theorem}\label{thm:regularity}
The critical curve $\kappa\mapsto \alpha_c(\kappa)$ is continuous, differentiable at $\kappa \neq \kappa_c$, and monotone increasing for $\kappa \geq \kappa_c$. Furthermore, 
 \begin{equation}\label{eq:asymptotic-kappac}
\alpha_c(\kappa) = \frac{2\sqrt{d}(d-1)^{3/4}}{d-2} \sqrt{\kappa - \kappa_c} + O_d(\kappa - \kappa_c),\quad\text{as }\kappa \downarrow \kappa_c
\end{equation}
and 
\begin{equation}\label{eq:asymptotic-infty}
\alpha_c(\kappa) = d^2 \log \bigg( \frac{d-1}{d-2} \bigg)(1+\kappa)^2 + O_d(1),\quad\text{as }\kappa\to\infty.
\end{equation}
\end{theorem}
Figure~\ref{fig:crit_curve} illustrates Theorem~\ref{thm:regularity} and indicates the percolative and non-percolative regimes. 

The formula \eqref{eq:alpha-c-kappa} is also amenable to numerical simulation of $\alpha_c(\kappa)$. Table~\ref{tab:d_alpha} provides some numerical approximations to $\alpha_c(0)$ for different $d$'s; in particular, it clearly suggests that $\alpha_c(0)$ grows as $d-\frac32$, for $d\to\infty$. 
\begin{table}[t]
    \centering
    \begin{tabular}{c|c}

$d$ & $\alpha_c(0)$ \\
\midrule
$3$ & $1.593856$  \\
$4$ & $2.576890$  \\
$5$ & $3.563572$  \\
$6$ & $4.553838$ \\
$7$	& $5.546569$ \\
$10$ & $8.532992$  \\
$50$ & $48.506672$ \\
$100$ & $98.503332$  \\
$200$ & $198.501693$	
\end{tabular}
    \caption{Numerical approximations to the solution of \eqref{eq:alpha-c-kappa} for $\kappa = 0$.}
    \label{tab:d_alpha}
\end{table}
In the next proposition, we give the asymptotic expansion of $\alpha_c(\kappa)$ as $d\to\infty$ for any $\kappa>-1$. It is a simple consequence of the asymptotic analysis of \eqref{eq:alpha-c-kappa}, and we leave its details to the interested reader. 
\begin{proposition}\label{prop:asymptotic-d}
For any $\kappa > -1$,
\[
\alpha_c(\kappa) = (1+\kappa)^2(d+\frac{3}{2}) -3 + O \Big( \frac{(1+\kappa)^2}{d} \Big), \quad \text{as }d\to\infty.
\]
\end{proposition}
\begin{remark}
For the loop percolation on $\Z^d$, it was proven in \cite[Theorem~1.7]{CS16} that $2d-6 + O(d^{-1})\leq \alpha_c(0)\leq 2d+\frac32 + O(d^{-1})$, as $d\to\infty$. In view of Proposition~\ref{prop:asymptotic-d}, it is natural to conjecture that on $\Z^d$, $\alpha_c(0) = 2d-\frac32 + O(d^{-1})$, as $d\to\infty$. 
\end{remark}
Using some standard results on insertion-tolerant percolation on non-amenable graphs, we also prove that the percolation probability is zero on the critical curve. 

\begin{theorem} \label{T:no_percolation}
For all $\kappa>-1$, 
\[
\theta(\alpha_c(\kappa), \kappa) = 0.
\]
\end{theorem}

\medskip

Our study of $\alpha_c(\kappa)$ is based on the analysis of the auxiliary threshold
\begin{equation}\label{def:alpha-sharp}
\alpha_\#(\kappa) = \sup \big\{ \alpha > 0 : \E_{\alpha, \kappa} \big[ |\cC(0)| \big] < \infty \big\},
\end{equation}
which is more directly accessible through the two-point function and the symmetries of the tree, and the following sharpness result.

\begin{theorem} \label{thm:sharpness}
For all $\kappa\geq \kappa_c$, 
\[
\alpha_c(\kappa) = \alpha_\#(\kappa).
\]
\end{theorem}
\begin{remark}
For the loop percolation on $\Z^d$, it is known that $\kappa_c = 0$, see \cite[Theorem~1.1 and Proposition~3.4]{CS16}, and it follows from standard arguments (for percolation models with rapidly decaying correlations) that $\alpha_c(\kappa) = \alpha_\#(\kappa)$ for all $\kappa > 0$. However, the equality does not hold for $\kappa=0$ in low dimensions $d=3,4$: $\alpha_\#(0) = 0 < \alpha_c(0)$, see \cite[Theorems~1.1 and 1.4]{CS16}. It is conjectured in \cite{CS16} that in high dimensions $d\geq 5$, the equality $\alpha_\#(0) = \alpha_c(0)$ does hold. 
\end{remark}

Finally, we identify the critical behavior of the susceptibility and the percolation probability. For $\kappa>\kappa_c$, the exponents take their mean-field values: the susceptibility diverges like $(\alpha_c-\alpha)^{-1}$ and the percolation probability grows linearly in $\alpha-\alpha_c$. These are the same exponents as for the Bernoulli percolation on $\mathbb T_d$; see, for example, Chapters~9 and~10 of \cite{Gri99}. At $\kappa=\kappa_c$, where $\alpha_c=0$, the phase transition is of a higher order with the percolation probability decaying quadratically in $\alpha-\alpha_c$.

\begin{theorem} \label{thm:crit_expo}
    If $\kappa > \kappa_c$, then
    \begin{equation} \label{eq:suscep_expo}
        \E_{\alpha,\kappa} \big[ |\cC(0) | \big] \asymp \frac{1}{\alpha_c - \alpha}, \quad \text{ as } \alpha \uparrow \alpha_c,
    \end{equation}
    and
    \begin{equation}
         \theta(\alpha, \kappa) \asymp \alpha-\alpha_c, \qquad \text{as } \alpha \downarrow \alpha_c. \label{eq:beta_non-crit}
    \end{equation}
    In the case $\kappa=\kappa_c$,
    \begin{equation}
        \theta(\alpha, \kappa_c) \asymp \alpha^2, \qquad \text{as } \alpha \downarrow \alpha_c =0. \label{eq:beta-crit}
    \end{equation}
\end{theorem}

\smallskip

\subsection{Proof ideas}
By Theorem~\ref{thm:sharpness}, it suffices to prove Theorem~\ref{thm:alpha-c-kappa} for $\alpha_\#(\kappa)$. 

Let $(x_n)_{n \geq 0}$ be an infinite nearest neighbor simple path on $\mathbb{T}_d$ with $x_0=0$. Due to the tree structure, the analysis of $\mathbb E_{\alpha,\kappa}[|\cC(0)|]$ reduces to a one-dimensional problem on this infinite path.
We write $\{ u \leftrightarrow v\}$ for the event that vertices $u$ and $v$ are loop-connected. Then, by the symmetries of $\mathbb T_d$, 
\[
\E_{\alpha,\kappa} \big[ |\cC(0)| \big] = 
\sum\limits_{v\in\mathbb T_d} \P_{\alpha, \kappa}( 0 \leftrightarrow v) = 
1+ \sum_{n=1}^\infty d(d-1)^{n-1} \P_{\alpha, \kappa}( 0 \leftrightarrow x_n).
\]
By supermultiplicativity of the sequence $a_n=\P_{\alpha, \kappa}( 0 \leftrightarrow x_n)$ and the Fekete lemma, there exists $\lambda = \lambda(\alpha,\kappa)$ such that 
\[
\P_{\alpha,\kappa}(0 \leftrightarrow x_n) = \lambda^{n+o(n)}, \quad\text{as }n\to\infty, 
\]
which yields the following equivalent definition of $\alpha_\#(\kappa)$,
\begin{equation}\label{def:alpha-sharp-new}
\alpha_\#(\kappa) = \sup \Big\{ \alpha > 0 : \lambda(\alpha, \kappa) < \frac{1}{d-1} \Big\}.
\end{equation}
The value of $\alpha_\#(\kappa)$ is therefore characterized by the relation $\lambda\big(\alpha_\#(\kappa), \kappa\big) (d-1) = 1$.

\smallskip

In Theorem~\ref{T:two_point_order}, we prove that for $\alpha>0$ and $\kappa\geq\kappa_c$, $\lambda(\alpha,\kappa)$ is the unique solution to 
\[
(\lambda^{-1} -1) \sum_{n \geq 0} \Big( \big(1-\tau_\kappa^{2(n+1)}\big)^{-\alpha} -1 \Big) \lambda^{-n}  = 1,
\]
from which it follows that $\alpha_\#(\kappa)$ satisfies \eqref{eq:alpha-c-kappa}. 
(In fact, we show in Theorem~\ref{T:two_point_order} that $\P_{\alpha,\kappa}(0 \leftrightarrow x_n) = \psi\lambda^n(1+o(1))$, for some $\psi=\psi(\alpha,\kappa)$.)

\smallskip

To get the above formula for $\lambda(\alpha,\kappa)$, we observe that the sequence of right endpoints $x_s$ of the successive edges $(x_{s-1},x_s)$ on the ray $(x_n)_{n\geq 0}$ not traversed by any open loop is a renewal process and the event $\{0 \leftrightarrow x_n\}$ is precisely the event that the first renewal has not occurred before time $n$. The probability $q_n$ that the edge $(x_{n-1},x_n)$ is not traversed by any open loop (resp.\ that a renewal occurs at time $n$) can be explicitly computed
\[
q_n = \big(1 - \tau_\kappa^{2}\big)^\alpha \big(1 -\tau_\kappa^{2(n+1)}\big)^{- \alpha},
\]
see Lemma~\ref{l:qn}. The generating functions $Q(z) = \sum_{n\geq 0}q_n z^n$ of the renewal measure and $F(z)$ of the time of the first renewal obey the standard relation 
\[
F(z) = 1 - \frac{1}{Q(z)}, \quad |z|<1,
\]
see \eqref{eq:powerseries_FQ}. The exponential decay rate $\lambda$ of the coefficients of $F$ can then be extracted by analysing the singularities of $Q$, which is the most technical part of this paper, see Section~\ref{S:complex_ana} and particularly Lemma~\ref{L:num_zeros}.

\smallskip

\subsection{Literature review}
Poisson ensembles of Markovian loops appeared informally in the work of Symanzik \cite{Sym69} on representations of the $\phi^4$ Euclidean field. Lawler and Werner \cite{LW04} described the loop ensembles precisely, whose properties---notably, their connections with the Gaussian free field and the Schramm-Loewner Evolution---have since been extensively studied, see e.g.\ 
\cite{LT07,LJ11,LJ24,Lup16a,Lup19,SW12,Szn12,JLQ23}.

First results on the percolation of loop clusters are due to Sheffield and Werner \cite{SW12} in the setting of planar Brownian loops. Le Jan and Lemaire \cite{LJL13} considered loop percolation on general graphs and proved the existence of a non-trivial supercritical regime. 
An extensive study of the subcritical loop percolation without killing ($\kappa=0$) on $\mathbb Z^d$ in dimension $d\geq 3$ was done by Chang and Sapozhnikov \cite{CS16}, who, notably, (a) identified the critical exponents for the one-arm probability, two-point function and the tail of the cluster size in dimensions $d\geq 5$, indicating that large clusters in the subcritical regime typically contain a single big loop, and (b) obtained non-trivial bounds on the one-arm probability in dimensions $d=3,4$, indicating that the single big loop scenario fails in low dimensions. 
The precise asymptotic for the one-arm probability was recently computed by Vogel \cite{Vog26}. Using a coupling between the loop percolation and the Gaussian free field, Lupu \cite{Lup16a} proved that $\alpha_c(0)\geq \frac12$ for the loop percolation on $\mathbb Z^d$ in any dimension $d\geq 3$. 
Chang \cite{Cha17} proved an analogue of the Grimmett-Marstrand theorem (see \cite{Gri99}) for the loop percolation, which implies, in particular, local uniqueness of macroscopic clusters for all $\alpha>\alpha_c(0)$. The loop percolation in the halfplane of $\mathbb Z^2$ was studied in \cite{Lup16b}, and the vacant set of the loop percolation on $\mathbb Z^d$ in \cite{AS19}. 

While the critical regime of the loop percolation on $\mathbb Z^d$ is still poorly understood, a comprehensive understanding of the phase transition in the related loop percolation on the metric graph of $\mathbb Z^d$ has been achieved, notably, precise values of the critical intensity and critical exponents were obtained, see \cite{Lup16a,Wer21,CDL24,DPR23,DPR25,CD24,CD25} and the references therein. 
The structure of three dimensional Brownian loop clusters was recently studied in \cite{LJ26}.

\smallskip

\subsection{Outline of the paper}
In Section~\ref{sec:notation}, we introduce commonly used notation and collect some preliminary results about the intensity measure $\mu_\kappa$ and loop occupation probabilities. In Section~\ref{sec:decay_rate}, 
we prove the existence of the exponential decay rate $\lambda(\alpha,\kappa)$ of the two-point function by relating the latter to the distribution of the holding times of a certain renewal process. The key Proposition~\ref{P:xi_order} about the decay rate of the holding times is proven in Section~\ref{S:complex_ana} with tools from complex analysis. Theorem~\ref{thm:sharpness} is proven in Section~\ref{sec:sharpness-proof}, and Theorems~\ref{thm:alpha-c-kappa}, \ref{thm:regularity} and \ref{T:no_percolation} are proven in Section~\ref{sec:proofs}. Finally, in Section~\ref{S:crit_exp}, we prove Theorem~\ref{thm:crit_expo}.

\smallskip

\section{Notation and preliminary results}\label{sec:notation}

\subsection{Notation}
Let $\mathbb T_d = \big(V(\mathbb T_d), E(\mathbb T_d)\big)$ be the $d$-regular tree rooted at the origin $0$. We write $d(u,v)$ for the graph distance between vertices $u$ and $v$ in the tree. We denote by $B(u,r)$ the closed ball of radius $r$ centered at $u$ with respect to the metric $d(\cdot,\cdot)$ and by $\partial B(u,r)$ the interior boundary of $B(u,r)$, 
$\partial B(u,r) = \{ v \in V(\mathbb{T}_d) : d(u,v) = r\}$. We write $B(r)$ for $B(0,r)$ and $\partial B(r)$ for $\partial B(0,r)$. Note that $|\partial B(r)| = d(d-1)^{r-1}$. 
Let $(x_i)_{i \in \Z}$ be a bi-infinite nearest neighbor simple path on $\mathbb T_d$ with $x_0 = 0$.

\smallskip

For a loop $\ell$, we write $v\in\ell$ if $\ell$ visits the vertex $v$, $v\notin\ell$ if $\ell$ does not visit $v$, $\ell\cap A\neq \emptyset$ if $\ell$ visits a vertex in the vertex subset $A$, and $\ell\subset A$ if all the vertices visited by $\ell$ are contained in $A$. For two sets of vertices $A_1$ and $A_2$, we write $A_1\xleftrightarrow{\ell} A_2$ if the loop $\ell$ intersects both $A_1$ and $A_2$. If one of the two sets is a singleton $\{v\}$, then we omit the brackets from the notation, in particular, we write $u\xleftrightarrow{\ell} v$ if $\ell$ visits $u$ and $v$. 

\smallskip

We denote by $\mathcal L_{\alpha,\kappa}$ the Poisson point process of loops with intensity $\alpha\mu_\kappa$. Loops that are present at least once in $\mathcal{L}_{\alpha,\kappa}$ are called \emph{open}. 
We write $\{u\leftrightarrow v\}$ for the event that 
there exists a sequence of open loops $(\ell_1,\ldots,\ell_k)$ such that (a) $u\in\ell_1$, (b) $v\in\ell_k$, and (c) each pair of loops $\ell_{i-1}, \ell_i$ of the sequence shares a vertex, and we say that $u$ and $v$ are \emph{loop-connected}. If $u$ and $v$ are loop-connected by a sequence of open loops with all their vertices contained in a vertex set $A$, then we say that $u$ and $v$ are \emph{loop-connected in $A$} and denote this event by $\{u\xleftrightarrow{A} v\}$. 
We assume that $u$ is always loop-connected to itself and that $u$ is loop-connected in $A$ to itself iff $u\in A$. 
Finally, we call an edge $e$ \emph{open} if it is traversed by at least one open loop, and call it \emph{closed} otherwise. 

\smallskip

For two functions $f,g$, we write $f = O(g)$ if there exists a constant $C$ such that $f(n) \leq C g(n)$, 
and $f = o(g)$ if $f(n)/g(n) \to 0$ as $n \to \infty$. If the asymptotic notation depends on $\alpha$ and/or $\kappa$, we indicate this by a subscript, e.g.\ $o_{\alpha,\kappa}(1)$.

\subsection{Green's function}

We define the Green generating function for the simple random walk on $\mathbb{T}_d$ as
\[
G(u,v;z) = \sum_{n=0}^\infty p_n(u,v) z^n,
\]
where $p_n(u,v)$ is the probability that a simple random walk started at $u$ visits $v$ after $n$ steps, and write
\[
G^\kappa(u,v) = G\Big(u,v; \frac{1}{1+\kappa}\Big),
\]
for $\kappa>-1$. For $i,j \in \Z$, we denote $G(x_i, x_j;z)$ (resp.\ $G^\kappa(x_i,x_j)$) by $G(i,j;z)$ (resp.\ $G^\kappa(i,j)$). 
By the symmetries of $\mathbb T_d$, if $d(u,v)=k$ then $G(u,v;z) = G(0,k;z)$ and $G^\kappa(u,v) = G^\kappa(0,k)$.
Recall the definition of $\tau_\kappa$ in \eqref{def:tau}.

\begin{lemma}\label{l:greens_f}
For any vertices $u$ and $v$, the Green function $G^\kappa(u,v)$ is finite if and only if $\kappa \geq \kappa_c$. Furthermore, for $\kappa\geq \kappa_c$, 
\begin{equation} \label{eq:greens_f}
G^\kappa(u,v) = \frac{\tau_\kappa^{d(u,v)}}{1-\tau_\kappa (1+ \kappa)^{-1}}.
\end{equation}
\end{lemma}
\begin{proof}
The radius of convergence of the Green generating function $G(u,v;z)$ is the inverse of the spectral radius of the random walk on $\mathbb T_d$, 
\[
\rho(\mathbb T_d) = \lim_{n\to \infty} p_{2n}(0,0)^{\frac{1}{2n}} = 2\sqrt{d-1}/d,
\]
see e.g.\ \cite{Woe00}.
By \eqref{def:kappa-c}, $\rho(\mathbb T_d) = 1+\kappa_c$; thus, $G^\kappa(u,v)$ is finite if $\kappa > \kappa_c$ and infinite if $\kappa<\kappa_c$. 

It suffices to prove \eqref{eq:greens_f} for $u=0$ and $v=x_r$ with $r\geq 0$. By the symmetries of $\mathbb T_d$, for $|z|<1+\kappa_c$ and $r\geq 1$,  
\begin{align*}
G(0,r;z) &= \sum_{n\geq 1} \sum_v p_{1}(0,v) p_{n-1}(v,x_r) z^n \\
&= \frac{d-1}{d} \sum_{n\geq 1} p_{n-1}(0,x_{r+1}) z^n + \frac{1}{d} \sum_{n \geq 1} p_{n-1}(0,x_{r-1}) z^n \\
&=z \Big( \frac{d-1}{d} G(0,r+1;z) + \frac{1}{d} G(0,r-1;z) \Big),
\end{align*}
and for $r=0$, $G(0,0;z) = 1 + zG (0,1;z)$.

If $\kappa>\kappa_c$, the recursive equation for $z = 1/(1+\kappa)$ is precisely solved by \eqref{eq:greens_f}. If $\kappa=\kappa_c$, \eqref{eq:greens_f} follows by the monotone convergence theorem and the continuity of $\tau_\kappa$ in $\kappa$. 
 \end{proof}

\subsection{Loop percolation tools}

The following formula is a key to explicit computations of loop occupation probabilities, see \cite[Lemma~2.5]{CS16} or \cite[Proposition 18]{LJ11}. 
\begin{lemma}\label{L:loop_det}
Let $\kappa\geq \kappa_c$. For a finite subset of vertices $F$,
\begin{equation}\label{eq:loop_det}
    \mu_\kappa(\ell : \ell \cap F \neq \emptyset)
    =
    \log \det\big( G^\kappa|_{F} \big),
\end{equation}
where $G^\kappa$ is the Green function viewed as a matrix $(G^\kappa(x,y))_{x,y}$, and $G^\kappa|_{F}$ is its submatrix indexed by $F \times F$.
In particular, for $n$ distinct vertices $v_1, \ldots, v_n$,
\begin{equation}\label{eq:loop_det_2}
    \mu_\kappa\big( \ell : v_i \in \ell \text{ for } i \in \{1,\ldots,n\} \big)
    =
    \sum_{\substack{A \subset \{v_1,\ldots,v_n\} \\ A \neq \emptyset}}
    (-1)^{\# A + 1}
    \log \det\big( G^\kappa|_{A} \big).
\end{equation}
\end{lemma}

In what follows, we collect several applications of Lemma~\ref{L:loop_det}.

\begin{lemma} \label{L:single_loop}
Let $\alpha>0$ and $\kappa\geq\kappa_c$. For all distinct vertices $u$ and $v$,  
\begin{equation}\label{eq:single_loop_mu_kappa}
 \mu_\kappa(\ell\,:\, u \xleftrightarrow{\ell} v) = - \log( 1 - \tau_\kappa^{2d(u,v)}).
\end{equation}
In particular, 
\begin{equation}\label{eq:single_loop_LS}
        \P\big(\exists \ell\in\mathcal L_{\alpha,\kappa}\,:\,u \xleftrightarrow{\ell} v\big) = 1 - (1 - \tau_\kappa^{2d(u,v)})^\alpha \leq 2\alpha\, \tau_\kappa^{2d(u,v)}. 
\end{equation}
\end{lemma}
\begin{proof}
The equality in \eqref{eq:single_loop_LS} follows from \eqref{eq:single_loop_mu_kappa} and the definition of the Poisson process of loops, 
\[
\P\big(\exists \ell\in\mathcal L_{\alpha,\kappa}\,:\,u \xleftrightarrow{\ell} v\big) = 1 - \exp\big(-\alpha\mu_\kappa(\ell\,:\, u \xleftrightarrow{\ell} v)\big),
\]
and the upper bound follows from the inequalities (a) $1-(1-x)^\beta \leq \beta x$ for $x \in [0,1]$ and $\beta \geq 1$ and (b) $1-(1-x)^\beta \leq 2(1-\frac{1}{2^\beta})x \leq 2\beta x$ for $x \in [0,\frac12]$ and $\beta \in[0,1]$, and the fact that $\tau_\kappa^2\leq \frac{1}{d-1}\leq \frac12$. 

The equality \eqref{eq:single_loop_mu_kappa} is immediate from \eqref{eq:loop_det_2} and Lemma~\ref{l:greens_f}. Indeed, 
\begin{align*}
\mu_\kappa(\ell\,:\, u \xleftrightarrow{\ell} v) &= \mu_\kappa( \ell : u,v \in \ell)\\
&= \log G^\kappa(u,u) + \log G^\kappa(v,v) - \log\big(G^\kappa(u,u)G^\kappa(v,v) - G^\kappa(u,v)^2 \big) \\
&= - \log \Big( 1 - \frac{G^\kappa(u,v)^2}{G^\kappa(0,0)^2} \Big) 
\stackrel{\eqref{eq:greens_f}}= - \log( 1 - \tau_\kappa^{2d(u,v)}).
\qedhere
\end{align*}
\end{proof}
\begin{remark}\label{rem:loops-cover-T}
If $\kappa < \kappa_c$, one can prove in a similar way to  \cite[Proposition~3.4]{CS16} that 
$\mu_\kappa(\ell : u\xleftrightarrow{\ell}v) = \infty$ for all $u$ and $v$. In particular, for any $\alpha>0$, the set of loops from $\mathcal L_{\alpha,\kappa}$ that visit $0$ covers the tree $\mathbb T_d$ almost surely. 
 \end{remark}

\begin{lemma}\label{l:single_loop_2-1}
Let $\alpha>0$ and $\kappa\geq\kappa_c$. Let $u,v,w$ be vertices such that $v$ lies on the geodesic between $u$ and $w$. Then 
\begin{equation}\label{eq:single_loop_2-1}
\P(\exists \ell\in\mathcal L_{\alpha,\kappa}\,:\, u,v \in \ell\text{ and }w\notin\ell) = 1 - \big(1 - \tau_\kappa^{2d(u,v)}\big)^\alpha\big(1 - \tau_\kappa^{2d(u,w)}\big)^{-\alpha}.
\end{equation}
\end{lemma}
\begin{proof}
It suffices to show that 
\[
\mu_\kappa (\ell\,:\, u,v \in \ell\text{ and }w\notin\ell) = 
\log\Big(\big(1 - \tau_\kappa^{2d(u,v)}\big)^{-1}\big(1 - \tau_\kappa^{2d(u,w)}\big)\Big).
\]
By Lemma~\ref{L:single_loop},
\begin{align*}
    \mu_\kappa(\ell : u,v \in \ell, w \not\in \ell) 
    &=  \mu_\kappa(\ell : u,v \in \ell) - \mu_\kappa(\ell : u,v,w \in \ell) \\
    &=  \mu_\kappa(\ell : u,v \in \ell) -  \mu_\kappa(\ell : u,w \in \ell) \\
    &= \log(1 - \tau_\kappa^{2d(u,w)}) - \log(1 - \tau_\kappa^{2d(u,v)}),
\end{align*}
and the result follows. 
\end{proof}

\begin{lemma} \label{L:3_single_loop}
Let $\alpha>0$ and $\kappa\geq\kappa_c$. For any distinct vertices $u,v,w$, 
\[
\P(\exists \ell\in\mathcal L_{\alpha,\kappa}\,:\, u,v,w \in \ell) \leq 16\alpha\, \tau_\kappa^{d(u,v) + d(u,w) + d(v,w)}.
\]
\end{lemma}
\begin{proof}
If one of the vertices, say $v$, lies on the geodesic between $u$ and $w$, then by  \eqref{eq:single_loop_LS},
\begin{align*}
\P(\exists \ell\in\mathcal L_{\alpha,\kappa}\,:\, u,v,w \in \ell) 
&= \P(\exists \ell\in\mathcal L_{\alpha,\kappa}\,:\, u,w \in \ell)
\leq 2\alpha\, \tau_\kappa^{2 d(u,w)}\\
&=  2\alpha\, \tau_\kappa^{d(u,w) + d(u,v) + d(v,w)}.
\end{align*}
We may therefore assume that $u,v,w$ are distinct, $w=0$,
$d(0,u) = i \geq j = d(0,v)$, and the common ancestor $a$ of $u$ and $v$ satisfies $d(0,a)=k$ with $1\le k<j$.
It suffices to show that 
\[
\mu_\kappa( \ell : 0,u,v \in \ell) \leq - 8 \log(1 - \tau_\kappa^{2(i+j-k)}).
\]
Indeed, if such a bound holds, in the same way as \eqref{eq:single_loop_LS} follows from \eqref{eq:single_loop_mu_kappa}, we get
    \begin{align*}
    \P(\exists \ell\in\mathcal L_{\alpha,\kappa}\,:\, 0,u,v \in \ell) 
        &= 1 - \exp\big(-\alpha \mu_\kappa( \ell : 0,u,v \in \ell)\big)
        \leq 1 - (1 - \tau_\kappa^{2(i+j-k)})^{8 \alpha } \\
        &\leq  16\alpha\, \tau_\kappa^{2(i+j-k)} =16\alpha\, \tau_\kappa^{d(u,v) + d(u,w) + d(v,w)},
    \end{align*}
    as required.
    Using that $d(u,v) = i+j-2k$, we obtain by \eqref{eq:loop_det_2} that
    \begin{align}
        \mu_\kappa( \ell : 0,u,v \in \ell) &= 3 \log G^\kappa(0,0) \label{eq:sum_dets} \\
        &\quad - \log \det \begin{pmatrix} G^\kappa(0,0) & G^\kappa(0,i) \\ G^\kappa(0,i) & G^\kappa(0,0)  \end{pmatrix} \nonumber \\
        &\quad - \log \det \begin{pmatrix} G^\kappa(0,0) & G^\kappa(0,j) \\ G^\kappa(0,j) & G^\kappa(0,0)  \end{pmatrix} \nonumber \\
        &\quad - \log \det \begin{pmatrix} G^\kappa(0,0) & G^\kappa(0,i+j-2k) \\ G^\kappa(0,i+j-2k) & G^\kappa(0,0)  \end{pmatrix} \nonumber \\
        &\quad + \log \det \begin{pmatrix} G^\kappa(0,0) & G^\kappa(0,i) & G^\kappa(0,j) \\ G^\kappa(0,i) & G^\kappa(0,0) & G^\kappa(0,i+j-2k) \\ G^\kappa(0,j) & G^\kappa(0,i+j-2k) & G^\kappa(0,0) \end{pmatrix}. \nonumber
    \end{align}
   By Lemma~\ref{l:greens_f}, 
    \begin{align*}
        \mu_\kappa( \ell : 0,u,v \in \ell) &= -\log( 1 - \tau_\kappa^{2i})  - \log( 1 - \tau_\kappa^{2j}) -  \log( 1 - \tau_\kappa^{2(i+j-2k)}) \\
        &\qquad + \log \Big( 1 + 2 \tau_\kappa^{2(i+j-k)} - \tau_\kappa^{2i} - \tau_\kappa^{2j} - \tau_\kappa^{2(i+j-2k)}  \Big).
    \end{align*}
    Let $A = 1 - \tau_\kappa^{2i} - \tau_\kappa^{2j} - \tau_\kappa^{2(i+j-2k)}$. Since $d \geq 3$ and $i \geq j > k \geq 1$, we have 
\[
(1-\tau_\kappa^{2i})(1-\tau_\kappa^{2j})(1-\tau_\kappa^{2(i+j-2k)}) \geq A \geq \frac{1}{4}.
\]
Therefore, using that $\log(1+x) \leq x$ gives
\begin{align*}
        \mu_\kappa( \ell : 0,u,v \in \ell) &= \log\Big( \frac{A + 2 \tau_\kappa^{2(i+j-k)}}{(1-\tau_\kappa^{2i})(1-\tau_\kappa^{2j})(1-\tau_\kappa^{2(i+j-2k)})} \Big) \\
        &\leq \log \Big(1 + \frac{2 \tau_\kappa^{2(i+j-k)}}{A}\Big) \leq \frac{2 \tau_\kappa^{2(i+j-k)}}{A}.
\end{align*}
Finally, the inequality $x \leq - \log(1-x)$ yields
\[
\mu_\kappa( \ell : 0,u,v \in \ell) \leq -8\log(1 - \tau_\kappa^{2(i+j-k)}). \qedhere
\]
\end{proof}

\begin{lemma} \label{L:loop_3_point}
Let $\alpha>0$ and $\kappa\geq\kappa_c$. 
Let $u,v,w$ be three distinct vertices whose common ancestor is $0$. 
Then 
\begin{equation}\label{eq:loop_3_point}
\P\big(\exists\, \ell_1,\ell_2,\ell_3\in \mathcal L_{\alpha,\kappa}\,:\,0 \xleftrightarrow{\ell_1} u, 0 \xleftrightarrow{\ell_2} v, 0 \xleftrightarrow{\ell_3} w\big) \leq \big(16\alpha + 12\alpha^2 + 8\alpha^3\big) \tau_\kappa^{d(u,v) + d(u,w) + d(v,w)}.
\end{equation}
\end{lemma}
\begin{proof}
We estimate the probability in \eqref{eq:loop_3_point} by considering different cases according to which of the loops $\ell_1,\ell_2,\ell_3$ coincide. 

If $\ell_1=\ell_2=\ell_3$, by Lemma~\ref{L:3_single_loop}, 
\begin{align*}
\P\big(\exists\, \ell_1\in \mathcal L_{\alpha,\kappa}\,:\,0 \xleftrightarrow{\ell_1} u, 0 \xleftrightarrow{\ell_1} v, 0 \xleftrightarrow{\ell_1} w\big) 
&= \P\big(\exists\,\ell_1\in\mathcal L_{\alpha,\kappa}\,:\, u,v,w\in\ell_1\big)\\
&\leq 16\alpha\, \tau_\kappa^{d(u,v) + d(u,w) + d(v,w)}.
\end{align*}

If $\ell_1\neq\ell_2=\ell_3$, let $A = \{ \ell : 0,u \in \ell \}$, $B = \{\ell: v,w \in \ell\}$, and consider 
\[
L_{A,B} = \sum_{\ell_1 \in \mathcal L_{\alpha,\kappa}} \sum_{\ell_2 \in \mathcal L_{\alpha,\kappa} \setminus \{\ell_1\}} \mathds 1_{\ell_1 \in A} \mathds 1_{\ell_2 \in B}.
\]
By the Markov inequality, 
\begin{align*}
\P\big(\exists\, \ell_1\neq\ell_2\in \mathcal L_{\alpha,\kappa}\,:\,0 \xleftrightarrow{\ell_1} u, 0 \xleftrightarrow{\ell_2} v, 0 \xleftrightarrow{\ell_2} w\big) 
&=\P\big(\exists\, \ell_1\neq\ell_2\in \mathcal L_{\alpha,\kappa}\,:\,0 \xleftrightarrow{\ell_1} u, v \xleftrightarrow{\ell_2} w\big)\\
&= \P(L_{A,B}\geq 1) \leq \E[L_{A,B}].
\end{align*}
The Slivnyak-Mecke theorem from the Palm theory for general
Poisson point processes (see e.g.\ \cite[Theorem 3.3]{MW04}, where it is proved for Poisson point processes in $\R^d$, and \cite[Chapter 13.1]{DVJ08} for the theory of Palm distributions in general spaces) implies that 
\[
\E[L_{A,B}] = 
\E\big[\sum_{\ell \in \mathcal L_{\alpha,\kappa}} \mathds 1_{\ell \in A}\big]\,\E\big[\sum_{\ell \in \mathcal L_{\alpha,\kappa}}\mathds 1_{\ell \in B}\big]= 
\big(\alpha\mu_\kappa(A)\big)\big(\alpha\mu_\kappa(B)\big).
\]
Thus, by \eqref{eq:single_loop_mu_kappa}, 
\begin{align*}
\E[L_{A,B}] &= \alpha^2 \big(-\log(1-\tau_\kappa^{2d(u,v)})\big)\big(- \log(1-\tau_\kappa^{2d(0,w)})\big)\\
&\leq \alpha^2 \frac{\tau_\kappa^{2(d(0,u) + d(v,w))}}{(1-\tau_\kappa^2)^2}
\leq 4 \alpha^2 \tau_\kappa^{d(u,v) + d(u,w) + d(v, w)},
\end{align*}
where we used the bound $-\log(1-x) \leq x/(1-x)$ and $1-\tau_\kappa^{2} \geq (d-2)/(d-1) \geq 1/2$. Thus, 
\[
\P\big(\exists\, \ell_1\neq\ell_2=\ell_3\in \mathcal L_{\alpha,\kappa}\,:\,0 \xleftrightarrow{\ell_1} u, 0 \xleftrightarrow{\ell_2} v, 0 \xleftrightarrow{\ell_3} w\big) \leq 4\alpha^2\, \tau_\kappa^{d(u,v) + d(u,w) + d(v,w)}.
\]
The cases $\ell_2\neq\ell_3=\ell_1$ and $\ell_3\neq\ell_1=\ell_2$ follow from the previous one by renaming the vertices $u,v,w$. 

Finally, when all three loops $\ell_1,\ell_2,\ell_3$ are distinct, we proceed as in the previous case. By the Slivnyak-Mecke theorem, the expected number of tuples $(\ell_1,\ell_2,\ell_3)$ of pairwise distinct loops satisfying $0 \xleftrightarrow{\ell_1} u, 0 \xleftrightarrow{\ell_2} v, 0 \xleftrightarrow{\ell_3} w$ equals 
\[
\big(\alpha\mu_\kappa(\ell\,:\,0 \xleftrightarrow{\ell} u)\big)
\big(\alpha\mu_\kappa(\ell\,:\,0 \xleftrightarrow{\ell} v)\big)
\big(\alpha\mu_\kappa(\ell\,:\,0 \xleftrightarrow{\ell} w)\big),
\]
which, arguing as before, is at most 
$8 \alpha^3 \tau_\kappa^{2(d(0,u) + d(0,v) + d(0,w))}= 8 \alpha^3 \tau_\kappa^{d(u,v) + d(u,w) + d(v, w)}$. Hence
\[
\P\big(\exists\, \ell_1\neq\ell_2\neq\ell_3\in \mathcal L_{\alpha,\kappa}\,:\,0 \xleftrightarrow{\ell_1} u, 0 \xleftrightarrow{\ell_2} v, 0 \xleftrightarrow{\ell_3} w\big) \leq 8\alpha^3\, \tau_\kappa^{d(u,v) + d(u,w) + d(v,w)}.
\]

Putting the three cases together gives \eqref{eq:loop_3_point}.
\end{proof}

\begin{lemma} \label{L:single_loop_to_boundary}
    For all $\kappa \geq \kappa_c$, we have
\[
\mu_\kappa(\ell : 0 \in \ell, \ell \cap \partial B(r) \neq \emptyset) = - \log\Big(1- \frac{\tau^{2r}_\kappa d(d-1)^{r-1}}{1 + \sum_{k=1}^{r-1} (d-2)(d-1)^{k-1} \tau^{2k}_\kappa + (d-1)^r\tau^{2r}_\kappa} \Big).
\]
\end{lemma}
\begin{proof}
    Write $K_r$ for the matrix with $K_r(x,y) = \tau^{d(x,y)}_\kappa$ for $x,y \in \partial B(r)$, and let $\mathbf{1}_r$ be the vector with $| \partial B(r)|$ many ones. Lemmas~\ref{l:greens_f} and \ref{L:loop_det} give that
    \begin{align*}
         \mu_\kappa(\ell : 0 \in \ell, \ell \cap \partial B(r) \neq \emptyset) &=  \mu_\kappa(\ell : 0 \in \ell) +  \mu_\kappa(\ell :  \ell \cap \partial B(r) \neq \emptyset) \\
         &\hspace{2cm}-  \mu_\kappa(\ell :  \ell \cap (\partial B(r) \cup \{0\}) \neq \emptyset) \\
         &= \log G^\kappa(0,0) + \log \det \big( G^{\kappa}(0,0) \cdot K_r\big) \\
         &\hspace{2cm}- \log \det \Big( G^{\kappa}(0,0) \begin{pmatrix}
             K_r & \tau^{r}_\kappa \mathbf{1}_r \\
             \tau^r_\kappa \mathbf{1}^T_r & 1
         \end{pmatrix} \Big)  \\
         &= \log \det  K_r - \log \det \begin{pmatrix}
             K_r & \tau^{r}_\kappa \mathbf{1}_r \\
             \tau^r_\kappa \mathbf{1}^T_r & 1
         \end{pmatrix}.
    \end{align*}
    The block matrix has determinant equal to $\det(K_r)( 1 - \tau^{2r}_\kappa \mathbf{1}^T_r K_r^{-1} \mathbf{1}_r)$, and hence
\[
\mu_\kappa(\ell : 0 \in \ell, \ell \cap \partial B(r) \neq \emptyset) = - \log( 1 - \tau^{2r}_\kappa \mathbf{1}^T_r K_r^{-1} \mathbf{1}_r).
\]
The vector $\mathbf 1_r$ is an eigenvector of $K_r$, since for every $u \in \partial B(r)$ one has
\begin{equation}
(K_r \mathbf{1}_r)_u  = \sum_{v \in \partial B(r)} \tau^{d(u,v)}_\kappa = 1 + \sum_{k=1}^{r-1} (d-2)(d-1)^{k-1} \tau^{2k}_\kappa + (d-1)^r\tau^{2r}_\kappa  =: A_r. \label{eq:def_A_r}
\end{equation}
Therefore,
\[
\tau^{2r}_\kappa \mathbf{1}^T_r K_r^{-1} \mathbf{1}_r = \tau^{2r}_\kappa \frac{|\partial B(r)|}{A_r} = \tau^{2r}_\kappa \frac{d(d-1)^{r-1}}{A_r}, 
\]
and hence
\[
\mu_\kappa(\ell : 0 \in \ell, \ell \cap \partial B(r) \neq \emptyset) = - \log( 1 - \frac{\tau^{2r}_\kappa d(d-1)^{r-1}}{A_r}). \qedhere
\]
\end{proof}

\medskip

\section{Two-point function decay rate} \label{sec:decay_rate}
In this section, we identify the exponential decay rate $\lambda(\alpha,\kappa)$ of the two-point function $\P_{\alpha,\kappa}(0 \leftrightarrow x_r)$. While the existence of the logarithmic asymptotics of $\P_{\alpha,\kappa}(0 \leftrightarrow x_r)$ follows easily from the supermultiplicativity of the two-point function and the Fekete lemma, Theorem~\ref{T:two_point_order} gives a precise asymptotics of the two-point function and provides an equation for its exponential decay rate. 
\begin{theorem} \label{T:two_point_order}
    Let $\alpha > 0$ and $\kappa \geq \kappa_c$. There exist $\lambda = \lambda(\alpha, \kappa)$ and $\psi = \psi(\alpha,\kappa) > 0$ such that for all $r \geq 1$,
    \begin{equation}\label{eq:psi_two_point}
        \P_{\alpha,\kappa}( 0 \leftrightarrow x_r) = \psi\lambda^r\, (1 + o_{\alpha,\kappa}(1)). 
    \end{equation}
    Furthermore, $\lambda$ is monotone increasing in $\alpha$, and it is the unique solution in $(\tau^2_\kappa, 1)$ to
    \begin{equation} \label{eq:two_point_order}
         (\lambda^{-1} -1) \sum_{n \geq 0} \big( (1-\tau_\kappa^{2(n+1)})^{-\alpha} -1 \big) \lambda^{-n}  = 1.
    \end{equation}
\end{theorem}
\begin{remark}\label{R:lambda>tau^2}
Notice that $\lambda$ from Theorem~\ref{T:two_point_order} always satisfies $\lambda > \tau_\kappa^2$. In particular, by Lemma~\ref{L:single_loop} this implies that
\[
 \P\big(\exists \ell\in\mathcal L_{\alpha,\kappa}\,:\,0 \xleftrightarrow{\ell} v\,\big|\,0 \leftrightarrow v\big) \longrightarrow 0,\, \text{ as }\,d(0,v)\to\infty,
\]
so that for any $\alpha > 0$, the main contribution to the two-point function does not come from a single loop. In contrast, the two-point function of the subcritical loop percolation on $\mathbb Z^d$ in dimensions $d\geq 5$ is comparable with the probability of connection by a single loop, see \cite[Theorem~1.4]{CS16} (but not comparable in low dimensions $d=3,4$, see \cite[Theorems~1.5 and 1.6]{CS16}). 
\end{remark}

\smallskip

\begin{proof}[Proof of Theorem~\ref{T:two_point_order}]
Fix $\alpha>0$ and $\kappa\geq \kappa_c$. We begin by defining a renewal process of successive closed edges along the ray $(x_n)_{n\in\Z}$. 
A similar renewal process was considered in Section~3 of \cite{LJL13}. 
Let $Z_0=0$ and let $Z_n$ be the right endpoint of the $n$-th successive closed edge among the edges $(x_{i-1},x_i)$, $i\geq 1$. 

Notice that conditioned on an edge $(x_{s-1}, x_s)$ being closed, the sets of loops in $\mathcal L_{\alpha,\kappa}$ that intersect edges $(x_{i-1},x_i)$ with $i<s$ and of those that intersect edges $(x_{j-1},x_j)$ with $j>s$ are disjoint. Therefore, by the independence property of the Poisson point process and translation invariance, the increments $\xi_i=Z_i-Z_{i-1}$ are independent and satisfy
\[
\P(\xi_i = r) = 
\P_{\alpha,\kappa}( 0\leftrightarrow x_{r-1}, 0\centernot\leftrightarrow x_r \mid x_{-1} \centernot\leftrightarrow x_0), \quad\text{for } i\geq 2.
\]
Furthermore, conditionally on the edge $(x_{-1}, x_0)$ being closed, $\xi_1$ also has the same distribution. More precisely, if we define the measure 
\begin{equation}\label{def:Q}
Q(\cdot) = \P_{\alpha,\kappa}( \cdot \mid x_{-1} \centernot\leftrightarrow x_0), 
\end{equation}
then $(\xi_i)_{i \geq 1}$ are i.i.d.\ under $Q$ with the law governed by
\begin{equation}
     Q(\xi_1 \geq r+1) = Q(0 \leftrightarrow x_r). \label{eq:renew_xr_xi}
\end{equation}
In particular, if we can estimate the distribution of $\xi_1$, we can obtain bounds on connection probabilities under $Q$.

\smallskip

Let $q_n$ be the probability that there is a renewal at time $n$, 
\[
q_n:= Q((x_{n-1},x_n)\text{ is closed}) = Q(x_{n-1} \centernot\leftrightarrow x_n).
\]
\begin{lemma}\label{l:qn}
    For $n\geq 1$, 
    \begin{align}
    q_n = (1 - \tau_\kappa^{2})^\alpha (1 -\tau_\kappa^{2(n+1)})^{- \alpha}. \label{eq:def_q_n}
    \end{align}
\end{lemma}
\begin{proof} 
We have 
\begin{align*}
    q_n &= 1 - Q(x_{n-1} \leftrightarrow x_n)
    = 1 - \frac{\P_{\alpha,\kappa}(x_{n-1} \leftrightarrow x_n, x_{-1} \centernot\leftrightarrow x_0)}{\P(x_{-1} \centernot\leftrightarrow x_0)}\\ 
    &= 1- \frac{\P(\{\exists \ell\in\mathcal L_{\alpha,\kappa}\,:\,x_{n-1},x_n\in\ell\text{ and }x_{-1}\notin \ell\}\cap\{\nexists \ell \in\mathcal L_{\alpha,\kappa}\,:\, x_{-1},x_0\in \ell\})}{\P(\nexists \ell \in\mathcal L_{\alpha,\kappa}\,:\, x_{-1},x_0\in \ell)} \\
    &\stackrel{(*)}= 1- \P(\exists \ell\in\mathcal L_{\alpha,\kappa}\,:\,x_{n-1},x_n\in\ell\text{ and }x_{-1}\notin \ell)\\
    &\stackrel{\eqref{eq:single_loop_2-1}}=(1 - \tau_\kappa^{2})^\alpha (1 -\tau_\kappa^{2(n+1)})^{- \alpha},
\end{align*}
where ($*$) holds by the independence of the events $\{\exists \ell\in\mathcal L_{\alpha,\kappa}\,:\,x_{n-1},x_n\in\ell\text{ and }x_{-1}\notin \ell\}$ and $\{\nexists \ell \in\mathcal L_{\alpha,\kappa}\,:\, x_{-1},x_0\in \ell\}$ as they are defined by disjoint sets of loops in $\mathcal L_{\alpha,\kappa}$. 
\end{proof}

We use the method of generating functions to identify the law of $\xi_1$ from the renewal measure $(q_n)_{n\geq 0}$. For $z\in\C$, let 
\[
Q(z) = 
\sum_{n \geq 0} q_n z^n \quad \text{and} \quad 
F(z) = \sum_{n \geq 0} Q(\xi_1 = n) z^n = \E_Q[z^{\xi_1}].
\]
For $|z|<1$, $Q(z)$ and $F(z)$ converge and 
\begin{align*}
    Q(z) - 1 &=\sum_{n \geq 1 } Q( \exists k\geq 1\,:\,Z_k=n) z^n = \sum_{n \geq 1 } \sum_{k \geq 1} Q( \xi_1 + \cdots + \xi_k = n) z^n \\
    &=\sum_{k \geq 1 } \sum_{n \geq 1} Q( \xi_1 + \cdots + \xi_k = n) z^n 
    =\sum_{k \geq 1 } \E_Q[z^{\xi_1 + \cdots + \xi_k}] \\
    &=\sum_{k \geq 1 } \E_Q[z^{\xi_1}]^k = \sum_{k\geq 1} F(z)^k =     
    \frac{1}{1 - F(z)} -1.
 \end{align*}
Thus, for $|z|<1$,
\begin{equation}
    Q(z) = \frac{1}{1 - F(z)} \quad \text{and} \quad
    F(z) = 1 - \frac{1}{Q(z)}. \label{eq:powerseries_FQ}
\end{equation}

Using this power series representation, we will prove in Section~\ref{S:complex_ana} (with the help of some complex analysis) the following proposition about the decay rate of the holding times $\xi_i$.
\begin{proposition} \label{P:xi_order}
    Let $\alpha >0$ and $\kappa \geq \kappa_c$. There exist $\lambda = \lambda(\alpha, \kappa)$ and $ \psi_1 = \psi_1(\alpha,\kappa)$  such that 
\[
Q(\xi_1 = n) = \psi_1 \lambda^n\,(1 + o_{\alpha,\kappa}(1)).
\]
    Furthermore, $\lambda$ is the unique solution in $(\tau_\kappa^2, 1)$ to
    \begin{equation}
         (\lambda^{-1} -1) \sum_{n \geq 0} \Big( (1-\tau_\kappa^{2(n+1)})^{-\alpha} -1 \Big) \lambda^{-n}  = 1. \label{eq:sol_to_sum}
    \end{equation}
\end{proposition}

It remains to express the connection probabilities under $Q$ in terms of the original law $\P_{\alpha,\kappa}$. We have 
\begin{align}
     \P_{\alpha,\kappa}(0 \leftrightarrow x_r) &= \P_{\alpha,\kappa}(0 \leftrightarrow x_r, x_{-1} \centernot\leftrightarrow 0) + \P_{\alpha,\kappa}(0 \leftrightarrow x_r, x_{-1} \leftrightarrow 0) \nonumber \\
     &=\P_{\alpha,\kappa}( x_{-1} \centernot\leftrightarrow 0)Q(0 \leftrightarrow x_r) + \P_{\alpha,\kappa}(0 \leftrightarrow x_{r+1}) \nonumber \\
     &=  \P_{\alpha,\kappa}(x_{-1} \centernot\leftrightarrow 0) \sum_{k=r}^\infty Q(0 \leftrightarrow x_k). \label{eq:P_Q_sum}
 \end{align} 
By Proposition~\ref{P:xi_order}, 
\[
Q(0 \leftrightarrow x_k) = Q(\xi_1 \geq k+1) = \sum_{n=k+1}^\infty Q(\xi_1 = n) = \psi_1 \frac{\lambda}{1- \lambda} \lambda^{k}\,(1+o_{\alpha,\kappa}(1)), \,\,\text{as }k\to\infty.
\]
Thus, 
    \begin{align*}
        \P_{\alpha,\kappa}(0 \leftrightarrow x_r) &= \psi_1\frac{\lambda}{1- \lambda} \P_{\alpha,\kappa}(x_{-1} \centernot\leftrightarrow 0) \big(\sum_{k=r}^\infty  \lambda^k\big)\, (1+o_{\alpha,\kappa}(1))\\
        &\stackrel{\eqref{eq:single_loop_LS}}=\psi_1 \frac{\lambda}{(1- \lambda)^2} (1-\tau_\kappa^2)^\alpha\, \lambda^{r}\,(1+o_{\alpha,\kappa}(1)),\,\,\text{as }r\to\infty,
    \end{align*}
and \eqref{eq:psi_two_point} holds with 
\begin{equation}\label{eq:psi-psi1}
\psi(\alpha, \kappa) = \psi_1\frac{\lambda}{(1- \lambda)^2} (1-\tau_\kappa^2)^\alpha.
\end{equation}
The proof of Theorem~\ref{T:two_point_order} is completed by noting that $\lambda$ is uniquely determined as the solution to \eqref{eq:sol_to_sum}, and that the left-hand side of \eqref{eq:sol_to_sum} is monotone increasing in $\alpha$ whenever $\lambda^{-1} > 1$. 
\end{proof}

\begin{remark}\label{rem:longrange}
Even though the two-point function $\P_{\alpha,\kappa}(u \leftrightarrow v)$ decays exponentially in $d(u,v)$, the loop percolation exhibits strong spatial correlations for $\kappa\leq 0$. For example, it follows from Lemma~\ref{L:single_loop} that 
\[
\inf_{r\geq 1}\,\mu_0(\ell\,:\,\partial B(r) \xleftrightarrow{\ell} \partial B(2r)) >0.
\]
Thus,
\[
\inf_{r\geq 1}\,\P\big(\exists\,\ell\in\mathcal L_{\alpha,0}\,:\, \partial B(r) \xleftrightarrow{\ell} \partial B(2r)\big) >0.
\]
(For $\kappa>0$, the measure $\mu_\kappa$ of loops intersecting $\partial B(r)$ and $\partial B(2r)$ goes to $0$ exponentially fast, and so does the probability that such a loop exists in $\mathcal L_{\alpha,\kappa}$.)
\end{remark}

\smallskip

\section{Generating functions and growth rate} \label{S:complex_ana}
The goal of this section is to prove Proposition~\ref{P:xi_order}. The strategy is to extract the asymptotic behavior of the coefficients of the generating function $F(z)$ (from \eqref{eq:powerseries_FQ}) by analyzing the singularities of $Q(z)$. To this end, we extend the identity
\[
F(z) = 1 - \frac{1}{Q(z)}
\]
beyond $|z|<1$ to a larger domain in the complex plane. Throughout this section, we write $\tau$ for $\tau_\kappa$. 

\subsection{Analytic extension}
Denote by $B_r = \{ z \in \C : |z| < r \}$ the open ball of radius $r$ centered at $0$ in the complex plane. Recall the definition of $q_n$ from \eqref{eq:def_q_n} and let 
    \begin{equation} \label{eq:def_q^*}
        q^* = (1 - \tau^{2})^\alpha.
    \end{equation}
    Note that $q_n \downarrow q^*$. Consider 
\[
H(z) := \sum_{n\geq 0} (q_n - q^*) z^n = q^* \sum_{n \geq 0} \Big( (1 - \tau^{2(n+1)} )^{-\alpha} - 1\Big) z^n.
\]
This power series converges for $z \in B_{1/\tau^2}$. 

Let 
\[
D:= B_{1/\tau^{2}} \setminus \{ 1\} 
\]
and define for $z \in D$ the analytic extension of $Q(z)$ beyond the unit disk as  
\begin{equation}
    \widetilde{Q}(z) := \frac{q^*}{1-z} +H(z). \label{eq:def_tilde_Q}
\end{equation}
Notice that $\widetilde{Q}(z)$ is an analytic function on $D \subset \C$, and that for $z \in B_1$ the series $Q(z)$ agrees with $\widetilde{Q}(z)$ since
\[
Q(z) = \sum_{n\geq 0} q_n z^n =  \sum_{n\geq 0} (q^* + q_n - q^*) z^n = \sum_{n\geq 0} q^* z^n + H(z) = \frac{q^*}{1-z} + H(z),
\]
which is well defined for all $z \in B_1$. 

\smallskip

Denote by $Z_{\widetilde{Q}} = \{z \in D : \widetilde{Q}(z) = 0\}$ the zeros of $\widetilde{Q}(z)$ in $D$. Consider now
\[
\widetilde{F}(z) = 1 - \frac{1}{\widetilde{Q}(z)} \qquad \forall z \in D \setminus Z_{\widetilde{Q}},
\]
and define $\widetilde{F}(1) = 1$ (the point $z=1$ corresponds to a removable singularity of $1/\widetilde{Q}(z)$). Equivalently, we may write
\[
\widetilde{F}(z) = 1 - \frac{1-z}{q^* + (1-z) H(z)} \qquad \forall z \in B_{1/\tau^2} \setminus Z_{\widetilde{Q}},
\]
where $q^* + (1-z) H(z)$ is analytic in $B_{1/\tau^2}$. Note that the solutions of $\widetilde{Q}(z) = 0$ coincide with the solutions of $q^* + (1-z) H(z) = 0$.

\begin{lemma} \label{L:num_zeros}
    There exists a real number $z_0 \in (1, 1/\tau^{2})$ such that $Z_{\widetilde{Q}} = \{z_0\}$. Furthermore, $z_0$ is a simple zero of $\widetilde{Q}(z)$ and it is given as the unique solution of $z \in D$ to
    \begin{equation}
        (z-1) \sum_{n\geq 0}\Big((1- \tau^{2(n+1)})^{-\alpha}  - 1\Big) z^n = 1. \label{eq:sol_z}
    \end{equation}
\end{lemma}

\subsection{Proof of Proposition~\ref{P:xi_order}}
Before proceeding to the proof of Lemma~\ref{L:num_zeros}, let us discuss the consequences of $\widetilde{Q}(z)$ having only a single zero in $D$. The function $\widetilde{F}(z)$ is meromorphic in $B_{1/\tau^2}$, that is, it is analytic in $B_{1/\tau^2}$ except at isolated points where $\widetilde{F}(z)$ has a pole of finite order. Indeed, the zeros of $\widetilde{Q}(z)$ precisely correspond to the poles of $\widetilde{F}(z)$. In particular, the only (simple) pole of $\widetilde{F}(z)$ is at $z_0$. 

Since $\widetilde{Q}(z) = Q(z)$ for all $z \in B_1$, it follows that $\widetilde{F}(z) = F(z)$ in a neighbourhood of the origin. Therefore,  $\widetilde{F}(z)$ and $F(z)$ have the same Taylor expansion at $0$, i.e.\
\[
\widetilde{F}(z) = \sum_{n \geq 0} f_n z^n := \sum_{n \geq 0} Q(\xi_1 = n)z^n 
\]
for $z$ where this sum converges.  
We will use the following theorem (see also Remark~\ref{R:expansion_C}).

\begin{theorem}[Theorem~IV.10  of \cite{FS09}] \label{T:expand_mero}
Let $G(z)$ be a function meromorphic at all points of the closed disc $|z| \le R$, with poles at points $z_1, \ldots, z_m$, and Taylor expansion
\[
G(z) = \sum_{n\geq 0} g_n z^n
\]
around the origin.
Assume that $G(z)$ is analytic at all points of $|z| = R$ and at $z = 0$.
Then there exist $m$ polynomials $\{\Pi_j(x)\}_{j=1}^m$ such that
\[
g_n = \sum_{j=1}^{m} \Pi_j(n)z_j^{-n} + O(R^{-n}).
\]
Furthermore, the degree of $\Pi_j$ is equal to the order of the pole of $G(z)$ at $z_j$ minus one.
\end{theorem}

\begin{proof}[Proof of Proposition~\ref{P:xi_order}]
    By Lemma~\ref{L:num_zeros}, the function $\widetilde{Q}(z)$ has only a single zero at $z_0$. This zero is simple and is given by the solution to \eqref{eq:sol_z} (so that $\lambda$ will be equal to $1/z_0$). Consider an open ball centered at the origin with radius $R$ that is strictly between $z_0$ and $1/\tau^2$ (e.g.\ the midpoint). Note that $\widetilde{F}(z)$ has only a single pole (of order $1$) at $z_0$, $\widetilde{F}(z)$ is analytic on $|z| = R$, and it has a Taylor expansion
\[
\widetilde{F}(z) = \sum_{n \geq 0} f_nz^n=\sum_{n \geq 0} Q(\xi_1 = n) z^n
\]
at the origin. Applying Theorem~\ref{T:expand_mero} to $\widetilde{F}(z)$ with coefficients $f_n$, gives that there exists a degree $0$ polynomial (i.e.\ a constant) $\psi_1 = \psi_1(\alpha, \kappa)$ depending on $\widetilde{F}(z)$ (and thus on $\alpha$ and $\kappa$) with 
\[
Q(\xi_1 = n) = \psi_1\Big( \frac{1}{z_0} \Big)^{n} + O_{\alpha,\kappa}(R^{-n}) = \psi_1 \Big( \frac{1}{z_0} \Big)^{n}\,(1+o_{\alpha,\kappa}(1)). \qedhere
\]
\end{proof}

\begin{remark} \label{R:expansion_C}
We briefly explain Theorem~\ref{T:expand_mero} in the present setting, which also identifies the function $\psi_1$. Let $\gamma$ denote the positively oriented boundary of $B_R$, where $z_0 < R < 1/\tau^2$. By the residue theorem,
\[
\frac{1}{2i \pi } \int_{\gamma} \frac{\widetilde{F}(z)}{z^{n+1}} dz = \operatorname{Res}\Big( \frac{\widetilde{F}(z)}{z^{n+1}}, 0 \Big) + \operatorname{Res}\Big( \frac{\widetilde{F}(z)}{z^{n+1}}, z_0\Big),
\]
where $\operatorname{Res}(f,a)$ denotes the residue of $f$ at $a$. Rearranging and applying Cauchy's integral formula yields 
\[
f_n = \operatorname{Res}\Big( \frac{\widetilde{F}(z)}{z^{n+1}}, 0 \Big) =  -\operatorname{Res}\Big( \frac{\widetilde{F}(z)}{z^{n+1}}, z_0\Big) + \frac{1}{2i \pi } \int_{\gamma} \frac{\widetilde{F}(z)}{z^{n+1}} dz .
\]
Since $\widetilde{F}(z)$ is analytic on $\gamma$, the integral is $O(R^{-n}) = o(z_0^{-n})$. On the other hand, since 
\[
\widetilde{Q}(z) = \widetilde{Q}'(z_0) (z-z_0) + o(z-z_0), \quad \text{as } z \to z_0,
\]
and $z_0$ is a simple zero, we obtain
\[
\operatorname{Res}\Big( \frac{\widetilde{F}(z)}{z^{n+1}}, z_0\Big) = \lim\limits_{z \to z_0} (z-z_0) \frac{\widetilde{F}(z)}{z^{n+1}} =  - \frac{1}{z_0\widetilde{Q}'(z_0)} z_0^{-n}.
\]
Combining the above identities yields the following expression for the function $\psi_1$ in Proposition~\ref{P:xi_order}, 
\[
\psi_1 = \frac{1}{z_0 \widetilde{Q}'(z_0)} > 0.
\]
If $\alpha$ is bounded away from $0$ and $\infty$, then $z_0$ stays bounded away from $1$ and $1/\tau^2$, and consequently $\widetilde{Q}'(z_0)$ remains bounded. In particular, if $\alpha$ is restricted to some compact set $K \subset (0,\infty)$, the function $\psi_1(\alpha,\kappa)$ is bounded away from $0$ on $K$. Moreover, the difference $R-z_0$ can be uniformly bounded from below, so that the $o_{\alpha,\kappa}(1)$ term in Proposition~\ref{P:xi_order} can also be made uniform in $\alpha$ for compact subsets of $(0,\infty)$.
\end{remark}

\subsection{Proof of Lemma~\ref{L:num_zeros}}
We now prove Lemma~\ref{L:num_zeros}. The argument has two parts: we first show that $\widetilde{Q}(z)$ has a unique simple zero on the real interval $(1,1/\tau^2)$, and then we show that there are no other zeros in $D$.
Notice that for real $z$ with $z \in (1,1/\tau^2)$ we have
\begin{align*}
    \lim\limits_{z \downarrow 1} \widetilde{Q}(z) &= - \infty \\
    \lim\limits_{z \uparrow 1/\tau^2} \widetilde{Q}(z) &= + \infty \\
    \widetilde{Q}'(z) &= \frac{q^*}{(1-z)^2} +H'(z) > 0,
\end{align*}
where the last inequality follows from the fact that the coefficients of $H(z)$ are strictly positive. By the intermediate value theorem, there exists a unique $z_0 \in (1, 1/\tau^2)$ such that $\widetilde{Q}(z_0) = 0$. Since $\widetilde{Q}'(z_0) > 0$, this must be a simple (i.e.\ of order one) zero. Rearranging the equation $\widetilde{Q}(z_0)=0$ yields \eqref{eq:sol_z}.

To exclude any remaining zeros, it is convenient to rewrite $\widetilde{Q}(z)$ as a sum of geometric series. We claim that for $z \in D$ we have
    \begin{equation} \label{eq:new_Q}
        \widetilde{Q}(z) = q^* \sum_{k \geq 0} \binom{\alpha +k -1}{k}\frac{\tau^{2k}}{1 - z \tau^{2k}},
    \end{equation}
    where for $\beta \in \C$ we define $\binom{\beta}{k} = \beta(\beta -1) \ldots (\beta-k+1)/k!$ with $\binom{\beta}{0} = 1$. Since $|\tau^{2(n+1)}|<1$, the negative binomial series (or equivalently, an infinite Taylor series expansion) gives
    \begin{equation}
        q_n = q^* (1 - \tau^{2(n+1)})^{-\alpha} = q^* \sum_{k \geq 0} \binom{\alpha + k - 1}{k} \tau^{2k(n+1)}. \label{eq:expand_1-tau}
    \end{equation}
    Substituting this into the definition of $\widetilde{Q}(z)$, we obtain
    \begin{align*}
        \widetilde{Q}(z) &= \frac{q^*}{1-z} + q^* \sum_{n\geq 0} \Big(\big( \sum_{k \geq 0} \binom{\alpha + k - 1}{k} \tau^{2k(n+1)}\big) -1\Big) z^n \\
        &= \frac{q^*}{1-z} + q^* \sum_{n\geq 0} \sum_{k \geq 1} \binom{\alpha + k - 1}{k} \tau^{2k(n+1)}   z^n.
     \end{align*}
     Interchanging the order of summation gives, for $z \in D$ (so that $z \neq 1$ and $|\tau^{2k} z| < 1$ for $k \geq 1$), that
     \begin{align*}
         \widetilde{Q}(z) &=  \frac{q^*}{1-z} + q^* \sum_{k \geq 1} \binom{\alpha + k - 1}{k} \sum_{n \geq 0}  \tau^{2k(n+1)}   z^n \\
          &=  \frac{q^*}{1-z} + q^* \sum_{k \geq 1} \binom{\alpha + k - 1}{ k} \frac{\tau^{2k}}{1 - \tau^{2k} z},
     \end{align*}
     as required (the $k=0$ term is simply $q^*/(1-z)$). Note that interchanging the order of summation is justified since for $z \in B_{1/\tau^2}$, we have $|z|\tau^{2k} \le |z|\tau^2 < 1$ for all $k\ge 1$ and
    \begin{align*}
        \sum_{k \geq 1} \binom{\alpha + k - 1}{k}
        \sum_{n\geq 0} \tau^{2k(n+1)} |z|^n
        &= \sum_{k \geq 1} \binom{\alpha + k - 1}{k}
           \frac{\tau^{2k}}{1-|z|\tau^{2k}} \\
        &\le
           \Bigg(\sup_{k\ge 1}\frac{1}{1-|z|\tau^{2k}}\Bigg)
           \sum_{k \geq 1} \binom{\alpha + k - 1}{k}\tau^{2k} \\
        &=
           \frac{1}{1-|z|\tau^2}\Big((1-\tau^2)^{-\alpha}-1\Big),
    \end{align*}
    which is finite. This proves \eqref{eq:new_Q}.

     We split the rest of $D \setminus (1,1/\tau^{2})$ into 3 regions:
     \begin{align*}
         R_1 &=\{ z \in D: \Re(z) \in [0,1), \Im(z) = 0 \}, \\
         R_2 &=\{ z \in D : \Re(z) \in (-1/\tau^2, 0), \Im(z) = 0 \},  \\
         R_3 &=  \{ z \in D: \Im(z) \neq 0 \},
     \end{align*}
     where $\Re(z)$ and $\Im(z)$ denote the real and imaginary parts of $z$, respectively. Since $\widetilde{Q}(z) > 0$ for $z\in [0,1)$, there are no zeros in $R_1$. For $z$ in the second region $R_2$, it holds that $1 - z\tau^{2k} \geq 1$, so that each summand in \eqref{eq:new_Q} is strictly positive whenever $\alpha > 0$. For the last region $R_3$, let $z = x + iy$ with $y \neq 0$. Then the quantity
\[
\Im(\frac{1}{1-z\tau^{2k}}) = \Im( \frac{1 - x \tau^{2k} + i y \tau^{2k}}{(1 - x \tau^{2k} - i y \tau^{2k})(1 - x \tau^{2k} + i y \tau^{2k})}) = \frac{y \tau^{2k}}{(1-x \tau^{2k})^2 + y^2 \tau^{4k}},
\]
is not equal to zero and has the same sign as $y$ for every $k$. As the series in \eqref{eq:new_Q} is absolutely convergent, it follows that
\[
\Im(\widetilde{Q}(z)) = q^* \sum_{k \geq 0} \binom{\alpha + k - 1}{ k} \Im(\frac{\tau^{2k}}{1 - z \tau^{2k}}) \neq 0,
\]
so that $\widetilde{Q}(z) \neq 0$, and Lemma~\ref{L:num_zeros} is proven.\qed

\section{Proof of Theorem~\ref{thm:sharpness}}\label{sec:sharpness-proof}

In this section, we prove that $\alpha_\#(\kappa)=\alpha_c(\kappa)$ for all $\kappa\geq \kappa_c$. Since the inequality $\alpha_\#(\kappa)\leq \alpha_c(\kappa)$ is obvious (finiteness of the expected cluster size implies almost sure finiteness of the cluster), it suffices to show that $\alpha_\#(\kappa)\geq \alpha_c(\kappa)$ for $\kappa\geq \kappa_c$. We will show this by a second moment argument. The key technical ingredient is Lemma~\ref{L:3_cluster}, which gives an upper bound on the three-point function.

\smallskip

Fix $\alpha > \alpha_\#(\kappa)$, and denote by $N_r$ the random set of vertices in $\partial B(r)$ that are loop-connected to $0$. By the Paley–Zygmund inequality, 
\[
 \P_{\alpha,\kappa} \big( 0 \leftrightarrow \partial B(r) \big) = \P(|N_r| > 0) \geq \frac{\E_{\alpha,\kappa} [ |N_r| ]^2}{\E_{\alpha,\kappa} [ |N_r|^2]},
\]
where by Theorem~\ref{T:two_point_order},
\begin{equation}
\E_{\alpha,\kappa} [ |N_r| ] = \sum_{v\in\partial B(r)}\P_{\alpha,\kappa}(0\leftrightarrow v) \geq |\partial B(r)|\,\frac12\psi\lambda^r = \frac12 d(d-1)^{r-1}\,\psi\lambda^r, \label{eq:N_r_first_moment}
\end{equation}
for all $r$ large enough.
Thus, if 
\begin{equation}\label{eq:second-moment-upper}
\E_{\alpha,\kappa}[|N_r|^2]\leq C (d-1)^{2r}\lambda^{2r},
\end{equation}
for some $C=C(\alpha,\kappa)$, then 
\[
\theta(\alpha,\kappa) = \lim_{r\to\infty}\P_{\alpha,\kappa}\big( 0 \leftrightarrow \partial B(r) \big)>0,
\]
that is $\alpha\geq \alpha_c(\kappa)$. Since $\alpha>\alpha_\#(\kappa)$ is arbitrary, we obtain that $\alpha_\#(\kappa)\geq \alpha_c(\kappa)$. 

\smallskip

It remains to prove \eqref{eq:second-moment-upper}. We have 
\[
\E_{\alpha,\kappa} \big[|N_r|^2 \big] = \sum_{u,v \in \partial B(r)} \P_{\alpha,\kappa} \big( u,v \in \cC(0) \big).
\]
The next key lemma gives an upper bound on the probability in the right hand side.
\begin{lemma} \label{L:3_cluster}
Let $\alpha \geq 0$ and $\kappa \geq \kappa_c$. 
There exists a constant $D = D(\alpha, \kappa)$ such that for any $u$ and $v$, 
\[
\P_{\alpha,\kappa} \big( u,v \in \cC(0) \big) \leq D\,\lambda^{\frac12(d(u,v) + d(0,u) + d(0,v))}.
\]
\end{lemma}

Before proving Lemma~\ref{L:3_cluster}, we show how it implies \eqref{eq:second-moment-upper}. By Lemma~\ref{L:3_cluster},
\[
\E_{\alpha,\kappa} \big[|N_r|^2 \big] \leq D \sum_{u,v \in \partial B(r)} \lambda^{\frac12(d(u,v) + d(0,u) + d(0,v))} 
= D \lambda^{r} \sum_{u,v \in \partial B(r)} \lambda^{\frac12 d(u,v)}.
\]
Fix $u \in \partial B(r)$ and $k \in \{0,1,\ldots,r\}$. Let $v \in \partial B(r)$ be such that the common ancestor of $u$ and $v$ (with respect to $0$) lies at distance $k$ from $0$. Then $d(u,v) = 2(r-k)$. If $k=0$, there are $(d-1)^r$ such vertices $v$, while if $k=r$, there is exactly one. For $1 \leq k \leq r-1$, there are $(d-2)(d-1)^{r-1-k}\leq (d-1)^{r-k}$ such vertices. Summing over all $u \in \partial B(r)$ and using symmetry, we obtain
\begin{align}
\sum_{u,v \in \partial B(r)} \lambda^{\frac12 d(u,v)}
&= |\partial B(r)| \sum_{k=0}^r \sum_{\substack{v \in \partial B(r): \\ d(u,v) = 2(r-k)}} \lambda^{r-k}
\leq |\partial B(r)| \sum_{k=0}^r (d-1)^{r-k}\lambda^{r-k} \nonumber \\
&\leq |\partial B(r)| (d-1)^r\lambda^r \sum_{k=0}^\infty \big((d-1)\lambda\big)^{-k}. \label{eq:double_boundary_sum}
\end{align}
Since $\alpha>\alpha_\#(\kappa)$, by \eqref{def:alpha-sharp-new}, we have $\lambda(d-1) > 1$, hence the infinite sum above is bounded. All in all, as $|\partial B(r)| = d(d-1)^{r-1}$, the inequality \eqref{eq:second-moment-upper} follows.\qed 

\smallskip

It remains to prove Lemma~\ref{L:3_cluster}. 
\begin{proof}[Proof of Lemma~\ref{L:3_cluster}]
The statement is trivial for $\alpha=0$, so we assume $\alpha>0$. 

Let $a$ be the common ancestor of $u$ and $v$. 
By considering the maximal single loops outgoing from $a$ in the directions of $0$, $u$ and $v$ (see also Figure~\ref{fig:3-point}), we obtain that if $u,v\in\cC(0)$, then there exist vertices $a'$, $u'$ and $v'$ on the geodesic paths between $a$ and $0$, resp.\ $a$ and $u$, resp.\ $a$ and $v$, and loops $\ell_0,\ell_u,\ell_v\in\mathcal L_{\alpha,\kappa}$ such that the events 
\begin{align*}
&\{\exists \ell_0\in \mathcal L_{\alpha,\kappa}\,:\,a\xleftrightarrow{\ell_0}a'\}\cap \{a'\xleftrightarrow{\mathbb T_d\setminus \{a\}} 0\},\\
&\{\exists \ell_u\in \mathcal L_{\alpha,\kappa}\,:\,a\xleftrightarrow{\ell_u}u'\}\cap \{u'\xleftrightarrow{\mathbb T_d\setminus \{a\}} u\},\\
&\{\exists \ell_v\in \mathcal L_{\alpha,\kappa}\,:\,a\xleftrightarrow{\ell_v}v'\}\cap \{v'\xleftrightarrow{\mathbb T_d\setminus \{a\}} v\}
\end{align*}
occur simultaneously; see Figure~\ref{fig:3-point}.
\begin{figure}[t]
    \centering
    \includegraphics[width=0.5\textwidth]{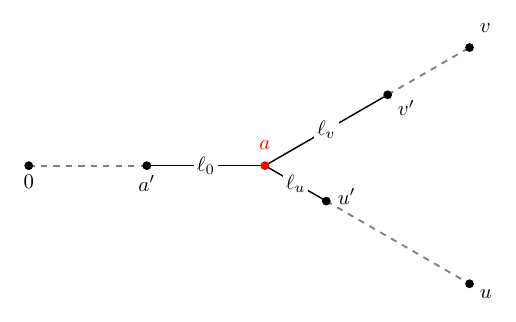}
    \caption{If $u,v \in \cC(0)$, then there are loops $\ell_0,\ell_u,\ell_v\in\mathcal L_{\alpha,\kappa}$ connecting the common ancestor $a$ of $u$ and $v$ to points $a',u',v'$ on the geodesic paths from $a$ to $0,u,v$, respectively. Moreover, the points $a',u',v'$ are connected, respectively, to $0,u,v$ via (disjoint sets of) loops that do not visit the vertex $a$.}
    \label{fig:3-point}
\end{figure}
Notice that the loops involved in the single-loop connections all contain vertex $a$ and the loops involved in the connections to $0$, $u$ and $v$ outside of $a$ are contained in different subtrees of $\mathbb T_d\setminus\{a\}$. Using the independence of the restrictions of the Poisson process to disjoint sets, the respective four events are independent. Hence 
\begin{multline*}
\P_{\alpha,\kappa}\big(u,v\in\cC(0)\big)\leq 
\sum_{a',u',v'} 
\P_{\alpha,\kappa}\big(
\exists \ell_0,\ell_u,\ell_v\in \mathcal L_{\alpha,\kappa}\,:\, a\xleftrightarrow{\ell_0}a', a\xleftrightarrow{\ell_u}u', a\xleftrightarrow{\ell_v}v'\big) \\
\P_{\alpha,\kappa}(a'\xleftrightarrow{\mathbb T_d\setminus \{a\}} 0)
\P_{\alpha,\kappa}(u'\xleftrightarrow{\mathbb T_d\setminus \{a\}} u)
\P_{\alpha,\kappa}(v'\xleftrightarrow{\mathbb T_d\setminus \{a\}} v).
\end{multline*}
By Lemma~\ref{L:loop_3_point},
\begin{align}
\P_{\alpha,\kappa}\big(
\exists \ell_0,\ell_u,\ell_v\in \mathcal L_{\alpha,\kappa}\,:\, a\xleftrightarrow{\ell_0}a', a\xleftrightarrow{\ell_u}u', a\xleftrightarrow{\ell_v}v'\big) 
&\leq 16(\alpha + \alpha^2 + \alpha^3)\tau_\kappa^{d(a',u')+d(u',v')+d(v',a')} \nonumber \\
&= 16(\alpha + \alpha^2 + \alpha^3)\tau_\kappa^{2(d(a,a')+d(a,u')+d(a,v'))} \label{eq:3_point_single_3_cluster},
\end{align}
and, by Theorem~\ref{T:two_point_order},  there exists $C=C(\alpha,\kappa)$ such that 
\begin{multline} \label{eq:3_different_loops}
\P_{\alpha,\kappa}(a'\xleftrightarrow{\mathbb T_d\setminus \{a\}} 0)
\P_{\alpha,\kappa}(u'\xleftrightarrow{\mathbb T_d\setminus \{a\}} u)
\P_{\alpha,\kappa}(v'\xleftrightarrow{\mathbb T_d\setminus \{a\}} v)  \\
\begin{aligned}[t]
&\leq \P_{\alpha,\kappa}(a'\leftrightarrow 0)
\P_{\alpha,\kappa}(u'\leftrightarrow u)
\P_{\alpha,\kappa}(v'\leftrightarrow v)
\leq C \psi^3\lambda^{d(a',0)+d(u',u)+d(v',v)}\\
&= C  \psi^3 \lambda^{\frac12(d(u,v) + d(0,u) + d(0,v))}\,
\lambda^{-(d(a,a')+d(a,u')+d(a,v'))}. 
\end{aligned} 
\end{multline}
Hence it suffices to prove that the sum 
\[
\sum_{a',u',v'} \big(\tau_\kappa^2\lambda^{-1}\big)^{d(a,a')+d(a,u')+d(a,v')}
\]
is finite. Let $\gamma = \tau_\kappa^2\lambda^{-1}$ and recall that $\gamma<1$ (see e.g.\ Remark~\ref{R:lambda>tau^2}). Thus, the above sum is bounded from above by 
\begin{equation}
    \sum_{k,l,m\geq 0} \gamma^{k+l+m} = \Big(\frac{1}{1-\gamma}\Big)^3<\infty. \label{eq:gamma_cubed}
\end{equation}

The proof of Lemma~\ref{L:3_cluster} is completed. 
\end{proof}

\begin{remark}
The bound in Lemma~\ref{L:3_cluster} is actually sharp. Indeed, let $u$ and $v$ be two vertices and denote by $a$ their common ancestor. 
By the FKG inequality and Theorem~\ref{T:two_point_order}, there exists $c=c(\alpha,\kappa)>0$ such that 
\[
\P_{\alpha,\kappa} ( u,v \in \cC(0) )\geq \P_{\alpha,\kappa}(u \leftrightarrow v) \P_{\alpha,\kappa}(0 \leftrightarrow a)
\geq c\,\lambda^{d(u,v) + d(0,a)} = c\,\lambda^{\frac12(d(u,v) + d(0,u) + d(0,v))}.
\]
In particular, 
$\P_{\alpha,\kappa} \big( u,v \in \cC(0) \big) \leq \frac{D}{c}\,\P_{\alpha,\kappa}(u \leftrightarrow v) \P_{\alpha,\kappa}(0 \leftrightarrow a)$.
\end{remark}

\smallskip

\section{Proofs of Theorems~\ref{thm:alpha-c-kappa}, \ref{thm:regularity} and \ref{T:no_percolation}}\label{sec:proofs}

\subsection{Proof of Theorem~\ref{thm:alpha-c-kappa}}

Let $\kappa>\kappa_c$. By Theorem~\ref{T:two_point_order} and \eqref{def:alpha-sharp-new}, $\alpha_\#(\kappa)$ is a unique positive solution to \eqref{eq:alpha-c-kappa}. By Theorem~\ref{thm:sharpness}, $\alpha_c(\kappa)=\alpha_\#(\kappa)$. 

\smallskip

It remains to show that $\alpha_c(\kappa)=0$ for all $\kappa\leq \kappa_c$. Since $\alpha_c(\kappa)$ is non-decreasing in $\kappa$, it suffices to prove that $\alpha_c(\kappa_c)=0$. 
By Theorem~\ref{thm:sharpness}, it suffices to prove that $\alpha_\#(\kappa_c)=0$.

Let $\alpha>0$. We need to show that $\E_{\alpha,\kappa_c}[|\cC(0)|]=\infty$. 
The expected size of the open cluster of the origin is at least as big as the expected number of vertices reachable by a single loop from $0$. Thus, by Lemma~\ref{L:single_loop}, 
\begin{align*}
\E_{\alpha, \kappa_c} \big[|\cC(0)| \big] &\geq 
\sum_v\P(\exists\,\ell\in\mathcal L_{\alpha,\kappa_c}\,:\,0\xleftrightarrow{\ell}v)
=
\sum_v\big(1 - (1 - \tau_{\kappa_c}^{2d(0,v)})^\alpha \big)\\
&=
1 + \sum_{r=1}^\infty d(d-1)^{r-1} 
\big(1 - (1 - \tau_{\kappa_c}^{2r})^\alpha \big).
\end{align*}
Since $\tau_{\kappa_c} = \frac{1}{\sqrt{d-1}}$, the latter series diverges. 
\qed

\subsection{Proof of Theorem~\ref{thm:regularity}}
Consider the function
\begin{equation} \label{def:F}
    F(\alpha, \kappa) =  \sum_{n \geq 0} \Big( (1-\tau_\kappa^{2(n+1)})^{-\alpha} -1 \Big) (d-1)^n  - \frac{1}{d-2},
\end{equation}
cf.\ \eqref{eq:alpha-c-kappa}.
One may verify that $F(\alpha,\kappa)$ is differentiable for $\alpha > 0$ and $\kappa > \kappa_c$, and 
\[
\frac{\partial}{\partial \alpha} F(\alpha, \kappa) > 0, \quad
\frac{\partial}{\partial \kappa} F(\alpha, \kappa) < 0.
\]
By the implicit function theorem, there exists a unique differentiable function $\varphi(\kappa)$ such that $F(\varphi(\kappa), \kappa) = 0$. By Theorem~\ref{thm:alpha-c-kappa}, $\alpha_c(\kappa)=\varphi(\kappa)$. 
Thus, by the implicit function theorem, $\alpha_c'(\kappa) > 0$ for $\kappa > \kappa_c$. 
The continuity of $\alpha_c(\kappa)$ at $\kappa=\kappa_c$ follows from the fact that $\alpha_c(\kappa_c)=0$ and \eqref{eq:asymptotic-kappac}, which we prove next. 

\smallskip

The proofs of the asymptotics \eqref{eq:asymptotic-kappac} and \eqref{eq:asymptotic-infty} are based on Taylor expansions of the sum in \eqref{def:F} and elementary algebraic manipulations, so we leave many of the technical details to the interested reader. 

We begin with the proof of \eqref{eq:asymptotic-kappac}. 
Fix $d \geq 3$ and write $\kappa = \kappa_c + \varepsilon$ for $\varepsilon>0$. By \eqref{def:tau} and the fact that $\tau_{\kappa_c} = \frac{1}{\sqrt{d-1}}$, 
\begin{align*}
\tau_\kappa^2 (d-1) &= 1 + \frac{2d}{\sqrt{d-1}}\varepsilon + \frac{d^2}{2(d-1)}\varepsilon^2 - \frac{(2\sqrt{d-1} + d\varepsilon)\sqrt{4d\sqrt{d-1}\,\varepsilon + d^2\varepsilon^2}}{2(d-1)}\\
&= 1 - 2\sqrt{d}(d-1)^{-1/4} \sqrt{\varepsilon} + O_d(\varepsilon),\,\,\text{as }\varepsilon\to 0.
\end{align*}
Thus, 
\begin{equation}\label{eq:expansion-tau}
\frac{1}{\tau_\kappa^2} = (d-1)\frac{1}{1-2\sqrt{d}(d-1)^{-1/4} \sqrt{\varepsilon} + O(\varepsilon)} = (d-1) + O_d(\sqrt{\varepsilon}),\,\,\text{as }\varepsilon\to 0.
\end{equation}

By \eqref{eq:alpha-c-kappa} and the inequality $(1-x)^{-\alpha}-1\geq \alpha x$, 
\[
\frac{1}{d-2}\geq \alpha_c(\kappa) \sum_{n \geq 0}\tau_\kappa^{2(n+1)} (d-1)^n = 
\alpha_c(\kappa) \frac{\tau_\kappa^2}{1-\tau_\kappa^2(d-1)}.
\]
Hence
\begin{equation}
\alpha_c(\kappa) \leq \frac{1 - \tau_\kappa^2(d-1)}{\tau_\kappa^2(d-2)} = O_d(\sqrt{\varepsilon}),\,\,\text{as }\varepsilon\to 0. \label{eq:alpha_close_kappa_c_upper}
\end{equation}
Similarly, using \eqref{eq:alpha-c-kappa} and the inequality $(1-x)^{-\alpha} -1 \leq \alpha x + 2 \alpha x^2$, valid for $x \in [0,1/2]$ and $\alpha \in [0,1]$ (which is applicable since $\alpha_c(\kappa) \to 0$ as $\varepsilon\to0$), we have 
\[
\frac{1}{d-2}\leq 
\alpha_c(\kappa) \frac{\tau_\kappa^2}{1-\tau_\kappa^2(d-1)}
+ 2\alpha_c(\kappa) \frac{\tau_\kappa^4}{1-\tau_\kappa^4(d-1)}
= \alpha_c(\kappa) \frac{\tau_\kappa^2}{1-\tau_\kappa^2(d-1)} + O_d(\sqrt{\varepsilon}).
\]
Combining this with \eqref{eq:alpha_close_kappa_c_upper} 
and using the expansion \eqref{eq:expansion-tau} yields
\[
\alpha_c(\kappa) = \frac{1 - \tau_\kappa^2(d-1)}{\tau_\kappa^2(d-2)} 
+O_d\big( \sqrt{\varepsilon} (1 - \tau_\kappa^2 (d-1)) \big)
= 
\frac{2\sqrt{d}(d-1)^{3/4}}{d-2} \sqrt{\varepsilon} + O_d(\varepsilon),\,\,\text{as }\varepsilon\to 0, 
\]
which is precisely \eqref{eq:asymptotic-kappac}. 

\smallskip

It remains to prove \eqref{eq:asymptotic-infty}. Recall that $\tau_\kappa\to0$ as $\kappa\to\infty$. Let 
\[
R_1( \alpha, \kappa) =  \sum_{n \geq 1} \Big( (1-\tau^{2(n+1)}_\kappa)^{-\alpha} -1 \Big) (d-1)^n > 0.
\]
By writing $F(\alpha,\kappa) =  (1-\tau_\kappa^2)^{-\alpha} - \frac{d-1}{d-2} + R_1(\alpha,\kappa)$ in \eqref{def:F}, we obtain the bound 
\[
(1-\tau_\kappa^2)^{-\alpha_c(\kappa)} = \frac{d-1}{d-2} - R_1(\alpha_c(\kappa),\kappa) \leq \frac{d-1}{d-2}.
\]
Using further that $x \leq - \log(1-x)$ gives
\begin{equation} \label{eq:App-alpha_sharp_up}
\alpha_c(\kappa) \leq \frac{1}{\tau_\kappa^2} \log \big(\frac{d-1}{d-2} \big).
\end{equation}
However, since
\begin{equation} \label{eq:App_tau_kappa}
\frac{1}{\tau_\kappa^2} = d^2(1+\kappa)^2 - 2(d-1) -(d-1)^2 \tau_\kappa^2,
\end{equation}
we obtain
\[
\alpha_c(\kappa) \leq d^2  \log \big(\frac{d-1}{d-2} \big)(1+\kappa)^2 +O_d(1),\,\,\text{as }\kappa\to\infty.
\]
Furthermore, by \eqref{eq:App-alpha_sharp_up}, $(1-\tau_\kappa^{2(n+1)})^{-\alpha_c(\kappa)} -1 = O_d(\alpha_c(\kappa) \tau_\kappa^{2(n+1)})$ for $n\geq1$, as $\kappa\to\infty$, 
hence
\[
R_1( \alpha_c(\kappa), \kappa) = O_d(\alpha_c(\kappa) \tau_\kappa^4) = O_d(\tau_\kappa^2)
\]
leading to 
\[
\alpha_c(\kappa) \geq -\frac{\log \big(\frac{d-1}{d-2} \big) + O_d(\tau_\kappa^2)}{\log(1-\tau_\kappa^2)}.
\]
Finally, with $-\log(1-x) \leq \frac{x}{1-x}$ and \eqref{eq:App_tau_kappa}, one also obtains
\[
\alpha_c(\kappa) \geq d^2  \log \big(\frac{d-1}{d-2} \big)(1+\kappa)^2 +O_d(1),\,\,\text{as }\kappa\to\infty,
\]
which is precisely \eqref{eq:asymptotic-infty}. \qed


\smallskip

\subsection{Proof of Theorem~\ref{T:no_percolation}}
We use general results about insertion-tolerant bond percolation on non-amenable transitive unimodular graphs, see \cite[Chapters~7 and 8]{LP16} for the relevant background.

The bond percolation induced by the Poisson loop ensemble $\mathcal L_{\alpha,\kappa}$ is insertion-tolerant, since it stochastically dominates the Bernoulli bond percolation induced by the restriction of $\mathcal L_{\alpha,\kappa}$ to the loops of length $2$, see the proof of \cite[Proposition~3.3]{CS16} (and the footnote on \cite[p.\ 238]{LP16}). The tree $\mathbb T_d$ is a non-amenable transitive unimodular graph. 

\smallskip

Let $N$ be the number of infinite open clusters. Since the event $\{N=k\}$ is a tail event, \cite[Proposition~7.1]{CS16} implies that $N$ is almost surely constant in $\{0,1,2,\ldots\}\cup\{\infty\}$. Moreover, 
since the loop percolation is insertion-tolerant, by \cite[Theorem~7.8]{LP16}, $N \in \{0,1,\infty\}$ almost surely. It therefore remains to exclude the cases $N=1$ and $N=\infty$.

\smallskip

Assume that $N=1$ almost surely. By the FKG inequality, 
\begin{align*}
\P_{\alpha,\kappa}(0\leftrightarrow v) &\geq 
\P_{\alpha,\kappa}(|\cC(0)|=\infty, |\cC(v)|=\infty)\geq
\P_{\alpha,\kappa}(|\cC(0)|=\infty)
\P_{\alpha,\kappa}(|\cC(v)|=\infty)\\
&= \theta(\alpha,\kappa)^2>0.
\end{align*}
On the other hand, by Theorem~\ref{T:two_point_order}, $\P_{\alpha,\kappa}(0\leftrightarrow v)\to 0$ as $d(0,v)\to\infty$, which contradicts the uniformly positive lower bound above. Thus, for any $\alpha>0$ and $\kappa\geq \kappa_c$, 
\[
N\in\{0,\infty\}\,,\,\text{ almost surely}.
\]
(It follows from Remark~\ref{rem:loops-cover-T} that for all $\alpha>0$ and $\kappa<\kappa_c$, $N=1$ almost surely.)

\smallskip

It remains to rule out the case $N=\infty$ for $\alpha=\alpha_c(\kappa)$. 
Let $N_r$ be the set of vertices in $\partial B(r)$ loop-connected to $0$, cf.\ Section~\ref{sec:sharpness-proof}. 
Assume that $N=\infty$ almost surely. By \cite[Proposition~8.32]{LP16}, 
any infinite open cluster has infinitely many ends, which particularly implies that 
\[
\P_{\alpha_c(\kappa),\kappa}\big(\lim_{r\to\infty}|N_r|=\infty\,\big|\,|\cC(0)|=\infty\big) =1.
\]
However, by Theorem~\ref{T:two_point_order} and the fact that $\lambda(\alpha_\#(\kappa),\kappa) = \frac{1}{d-1}$, 
\[
\sup_{r\geq 1} \E_{\alpha_c(\kappa), \kappa} \big[ |N_r| \big]  =  \sup_{r\geq 1} \E_{\alpha_\#(\kappa), \kappa} \big[ |N_r| \big] <\infty,
\]
which is in contradiction with the divergence of $|N_r|$. Therefore, the case $N = \infty$ is also impossible for $\alpha=\alpha_c(\kappa)$. \qed

\smallskip

\section{Critical exponents}\label{S:crit_exp}

In this section, we prove Theorem~\ref{thm:crit_expo}. We use the notation $f \asymp g$ to mean that both $f = O_{d,\kappa}(g)$ and $g = O_{d,\kappa}(f)$. The implicit constants in this notation depend on $d$ and $\kappa$, but not on $\alpha$. Furthermore, we write $c_d$ for a constant depending only on $d$ that may differ from line to line.

\subsection{Susceptibility for \texorpdfstring{$\kappa>\kappa_c$}{non-critical kappa}}

Fix $\kappa > \kappa_c$. By Theorem~\ref{thm:alpha-c-kappa}, $\alpha_c(\kappa) > 0$. 
By Theorem~\ref{T:two_point_order}, for any $\alpha<\alpha_c(\kappa)$, 
\[
\E_{\alpha,\kappa} \big[ |\cC(0)| \big] 
\asymp \psi \sum_{r =0}^\infty (d-1)^{r} \lambda(\alpha, \kappa)^r \asymp \frac{1}{1 - \lambda(\alpha, \kappa) (d-1)},
\]
where in the last step we used that $\psi\asymp 1$ for $\alpha$ bounded away from $0$ (see Remark~\ref{R:expansion_C}).
The function $\lambda(\alpha,\kappa)$ is differentiable in $\alpha$, and its derivative is obtained from \eqref{eq:two_point_order} by implicit differentiation. Thus, 
\[
\lambda(\alpha,\kappa) = \frac{1}{d-1} + \lambda'(\alpha_c(\kappa),\kappa)\big(\alpha-\alpha_c(\kappa)\big) + O_{d,\kappa}\big((\alpha-\alpha_c(\kappa))^2\big),\quad\text{as }\alpha\uparrow\alpha_c(\kappa).
\]
Plugging this into the above expression for the susceptibility gives \eqref{eq:suscep_expo}. \qed

\subsection{Percolation probability for \texorpdfstring{$\kappa>\kappa_c$}{non-critical kappa}}

Fix $\kappa>\kappa_c$ and let $\alpha>\alpha_c(\kappa)$. 

We begin with the proof of the lower bound in \eqref{eq:beta_non-crit}. Recall from Remark~\ref{R:expansion_C} that the $o_{\alpha, \kappa}(1)$ term in Theorem~\ref{T:two_point_order} can be made uniform. Therefore, from the argument around \eqref{eq:N_r_first_moment} and \eqref{eq:second-moment-upper}, we obtain that 
\[
\E_{\alpha, \kappa}[|N_r|] \geq \frac{1}{2}  \psi d (d-1)^{r-1} \lambda^r \quad\text{and}\quad
\E_{\alpha, \kappa}[|N_r|^2] \leq C(d-1)^{2r} \lambda^{2r},
\]
for $r$ large enough and 
\[
C = C(\alpha,\kappa) = c_d \psi^3(\alpha+\alpha^2 + \alpha^3)\frac{1}{1-\frac{1}{(d-1)\lambda}} \frac{1}{(1-\gamma)^3},
\]
as follows from \eqref{eq:double_boundary_sum}, \eqref{eq:3_point_single_3_cluster}, \eqref{eq:3_different_loops} and \eqref{eq:gamma_cubed}, where $\gamma = \tau_\kappa^2 \lambda^{-1}$ and $c_d$ is some constant depending only on $d$. 
By decreasing the value of $c_d$, we obtain with the Paley-Zygmund inequality that
\[
\theta(\alpha,\kappa) = \lim\limits_{r\to\infty} \P_{\alpha,\kappa}(|N_r|\geq 1) \geq  \frac{\E_{\alpha,\kappa}\big[|N_r|\big]^2}{\E_{\alpha,\kappa}\big[|N_r|^2\big]}\geq c_d  \frac{(1-\frac{1}{(d-1) \lambda})(1-\gamma)^3}{(\alpha + \alpha^2 + \alpha^3) \psi}\asymp 1-\frac{1}{(d-1) \lambda},
\]
since the terms $(1-\gamma)$ and $\psi$ are uniformly bounded for $\kappa > \kappa_c$ and $\alpha$ bounded away from $0$, see Remark~\ref{R:expansion_C}.
Lastly, as in the previous proof, we obtain that 
\begin{multline} \label{eq:expand_product}
1-\frac{1}{(d-1) \lambda(\alpha, \kappa)} \\
= 1-\frac{1}{(d-1) \big(\frac{1}{(d-1)} + \lambda'(\alpha_c(\kappa), \kappa)\big(\alpha-\alpha_c(\kappa)\big) + O((\alpha-\alpha_c(\kappa))^2 \big) }
\asymp \alpha-\alpha_c(\kappa),
\end{multline}
as $\alpha \downarrow \alpha_c(\kappa)$, which proves the lower bound in \eqref{eq:beta_non-crit}. 

\smallskip

We prove the upper bound in \eqref{eq:beta_non-crit} using a general result from \cite{Lyo92}, which relates percolation probability to the effective conductance of some electric network. We refer the reader to \cite{Lyo92} (and also to Chapters~5.3 and 5.4 of \cite{LP16}) for definitions.

For any adjacent vertices $v_{n-1}$ and $v_{n}$ at distance $n-1$ and $n$ from the origin, respectively, define the resistance of the edge $(v_{n-1},v_n)$ as
\[
r(v_{n-1}, v_{n}) =\frac{1}{\P_{\alpha, \kappa}(0 \leftrightarrow v_{n})} - \frac{1}{\P_{\alpha, \kappa}(0 \leftrightarrow v_{n-1})} 
\]
and let $C_{\text{eff}}(0 \leftrightarrow \infty)$ be the effective conductance from $0$ to infinity in the network with these resistances. 
The following lemma follows from \cite[Theorem~2.4]{Lyo92}. 
\begin{lemma}\label{l:Lyons}
For two vertices $x,y \in V(\mathbb{T}_d)$, denote by $x \wedge y$ their common ancestor. If there exists $M > 0$ such that for any $x,y \in V(\mathbb{T}_d)$ and any vertex set $A$ with the property that the removal of $x \wedge y$ would disconnect $x$ from $A$, 
\[
    \P_{\alpha, \kappa}( 0 \leftrightarrow x \mid 0 \leftrightarrow y, 0 \centernot\leftrightarrow A) \geq M  \P_{\alpha, \kappa}( 0 \leftrightarrow x \mid 0 \leftrightarrow x \wedge y),
\]
then 
\[
\theta(\alpha, \kappa) \leq \frac{4}{M} \frac{C_{\text{eff}}(0 \leftrightarrow \infty)}{1+C_{\text{eff}}(0 \leftrightarrow \infty)}.
\]
\end{lemma}
By Remark~\ref{R:expansion_C}, the $o_{\alpha, \kappa}(1)$ term in Theorem~\ref{T:two_point_order} can be made uniform, thus 
\[
r(v_{n-1}, v_{n}) = \frac{1}{\P_{\alpha, \kappa}(0 \leftrightarrow v_{n})} \big(1 -\frac{\P_{\alpha, \kappa}(0 \leftrightarrow v_{n})}{\P_{\alpha, \kappa}(0 \leftrightarrow v_{n-1})} \big) \asymp \lambda^{-n}.
\]
Hence, the effective conductance $C_{\text{eff}}(0 \leftrightarrow \infty)$ satisfies
\begin{equation} \label{eq:C_eff}
    C_{\text{eff}}(0 \leftrightarrow \infty) \asymp \Big( \sum_{n=1}^\infty  \frac{r(v_{n-1}, v_{n})}{(d-1)^n} \Big)^{-1} \asymp 1 - \frac{1}{(d-1) \lambda},
\end{equation}
which, as seen in \eqref{eq:expand_product}, is of order $\alpha - \alpha_c(\kappa)$ as $\alpha \downarrow \alpha_c(\kappa)$. Together with Lemma~\ref{l:Lyons}, this implies the mean-field upper bound in \eqref{eq:beta_non-crit}. Thus, it suffices to prove the assumption of Lemma~\ref{l:Lyons}. 

\smallskip

Fix vertices $x,y$ and set $A$ as in Lemma~\ref{l:Lyons} and let $z$ be the first vertex on the geodesic from $x \wedge y$ to $x$. If the loop $\ell_{x\wedge y, z}$ consisting of just the edge between $x \wedge y$ and $z$ is open, and $z$ is connected to $x$ using loops not containing $x \wedge y$, then $x \wedge y$ is connected to $x$. Furthermore, the occurrence of these two events is independent of one another and of the events $\{ 0 \leftrightarrow y\}$ and $\{ 0 \centernot \leftrightarrow A\}$ since they depend on loops intersecting different vertex sets. Thus,
\[
    \P_{\alpha, \kappa}( 0 \leftrightarrow x \mid 0 \leftrightarrow y, 0 \centernot\leftrightarrow A) \geq \P_{\alpha, \kappa}( \ell_{x\wedge y, z} \text{ is open}) \P_{\alpha, \kappa}( z \xleftrightarrow{\mathbb T_d\setminus \{x \wedge y\}} x ).
\]
As $\alpha \geq \alpha_c(\kappa) > 0$, the probability of the loop $\ell_{x\wedge y, z}$ being open is bounded away from $0$. Furthermore,
\[
    \P_{\alpha, \kappa}( z \xleftrightarrow{\mathbb T_d\setminus \{x \wedge y\}} x ) \geq \P_{\alpha, \kappa}( z \xleftrightarrow{} x \mid z \centernot\leftrightarrow x \wedge y ) \P_{\alpha, \kappa}(z \centernot\leftrightarrow x \wedge y ).
\]
For, say, $\alpha \leq 2 \alpha_c(\kappa)$, the probability that the edge between $z$ and $x \wedge y$ is not open is bounded from below by some constant depending only on $\kappa$ and $d$. Lastly, by Proposition~\ref{P:xi_order},
\[
    \P_{\alpha, \kappa}( z \xleftrightarrow{} x \mid z \centernot\leftrightarrow x \wedge y ) = Q(0 \leftrightarrow x_{d(z,x)})\asymp \lambda^{d(z,x)} \asymp \lambda^{d(x \wedge y,x)}
\]
so that there exists a constant $c_{d,\kappa}$ depending only on $d$ and $\kappa$ with 
\[
    \P_{\alpha, \kappa}( 0 \leftrightarrow x \mid 0 \leftrightarrow y, 0 \centernot\leftrightarrow A) \geq c_{d,\kappa} \lambda^{d(x \wedge y, x)} \asymp \P_{\alpha, \kappa}( 0 \leftrightarrow x \mid 0 \leftrightarrow x \wedge y).
\]
Thus, the assumption of Lemma~\ref{l:Lyons} is verified. \qed

\subsection{Percolation probability for \texorpdfstring{$\kappa=\kappa_c$}{critical kappa}}

We begin with the proof of the lower bound in \eqref{eq:beta-crit}. We will use the second moment method, as in the case $\kappa>\kappa_c$, however for $\kappa = \kappa_c$ and $\alpha \downarrow \alpha_c(\kappa_c) = 0$ the constant $\psi$ and the 
$o_{\alpha, \kappa}(1)$ term in Theorem~\ref{T:two_point_order} are no longer uniform. 
We begin with two lemmas, which establish the order of $\lambda$ and $\psi$ as $\alpha\downarrow 0$. 
\begin{lemma}\label{l:lambda-kappa-critical}
For $\alpha>0$, 
\begin{equation}\label{eq:crit_lambda_asymp}
\lambda(\alpha,\kappa_c) - \frac1{d-1}\asymp \alpha.
\end{equation}
\end{lemma}
\begin{proof}
Recall the expression for $q_n$ in \eqref{eq:def_q_n} and that $Q(z) = \sum_{n\geq 0} q_n z^n$. For $\alpha$ bounded from above, say, $\alpha \leq 1$, by the mean value theorem, 
\begin{equation}
    q_{n-1} - q_n \asymp \alpha \tau^{2n}_{\kappa_c} = \frac{\alpha}{(d-1)^{n}} \qquad \forall n \geq 1. \label{eq:q-diff-crit}
\end{equation}
The expression for the analytic extension of $Q(z)$ in \eqref{eq:def_tilde_Q} gives
\[
(1-z)\widetilde{Q}(z) = q^{*} + (1-z) H(z) = 1- \sum_{n \geq 1} (q_{n-1}-q_{n})z^n.
\]
In particular, since $\lambda^{-1}$ is the unique zero of $\widetilde{Q}(z)$, we obtain
\[
1 = \sum_{n \geq 1} (q_{n-1}-q_{n})\lambda^{-n} \asymp \alpha \sum_{n\geq 1} \frac{1}{(d-1)^n\lambda^n} = \alpha \Big(\frac{1}{1 - \frac{1}{(d-1)\lambda}} -1 \Big).
\]
Hence $1-\frac{1}{(d-1) \lambda} \asymp\alpha$ and the result follows. 
\end{proof}

\begin{lemma}\label{L:two_point_critical}
    There exists a constant $c_{d}$ depending only on $d$ such that
\[
\frac{1}{c_{d}} \alpha \lambda^r \leq\P_{\alpha, \kappa_c}(0 \leftrightarrow x_r) \leq c_{d} \alpha  \lambda^r,
\]
for all $r \geq 1$ and $0 \leq \alpha \leq 1$.
\end{lemma}
\begin{proof}
We use the notation from the proof of Theorem~\ref{T:two_point_order}. 
From the definition of the first renewal step, 
\begin{align*}
    \sum_{r \geq 0} Q(0 \leftrightarrow x_r) z^r &= \sum_{r \geq 0} \sum_{n \geq r+1}Q(\xi_1 =n) z^r = \sum_{n\geq 1} Q(\xi_1 =n) \sum_{r =0}^{n-1} z^r \\
    &=  \sum_{n\geq 1} Q(\xi_1 =n) \frac{1-z^n}{1-z}  = \frac{1-F(z)}{1-z} = \frac{1}{(1-z)Q(z)},
\end{align*}
or, equivalently,
\begin{equation} \label{eq:compare_coeff}
     (1-z)Q(z)\sum_{r \geq 0} Q(0 \leftrightarrow x_r) z^r =1.
\end{equation}
Comparing coefficients of $z^r$ in \eqref{eq:compare_coeff} yields for $r \geq 1$
\[
Q(0 \leftrightarrow x_r) = \sum_{n=1}^r (q_{n-1} - q_n) Q(0 \leftrightarrow x_{r-n}) = q_{r-1} - q_{r} + \sum_{n=1}^{r-1} (q_{n-1} - q_n) Q(0 \leftrightarrow x_{r-n}).
\]
We claim by induction that $Q(0 \leftrightarrow x_r) \leq c_{d} \alpha  \lambda^r$ whenever $c_d$ is large enough, where the $r=1$ case is obvious. 
Namely, from \eqref{eq:q-diff-crit} and \eqref{eq:crit_lambda_asymp} we obtain
\begin{align}
     Q(0 \leftrightarrow x_{r}) &\leq q_{r-1} - q_{r} + c_d \alpha \lambda^{r} \Big(\sum_{n = 1}^{r-1} (q_{n-1} - q_n) \lambda^{-n} \Big) \nonumber\\
     &= q_{r-1} - q_{r} + c_d \alpha \lambda^{r} \Big(1 -\sum_{n\geq r} (q_{n-1} - q_n) \lambda^{-n} \Big) \nonumber \\
     &\leq c_d \alpha \lambda^{r}+C \frac{\alpha}{(d-1)^{r}} - C' c_d \alpha^2 \lambda \frac{(d-1)^{-r}}{(d-1)\lambda - 1} \nonumber \\
     &\leq c_d \alpha \lambda^{r}+C \frac{\alpha}{(d-1)^{r}} - C'' c_d \lambda \frac{\alpha }{(d-1)^{r}}, \label{eq:induction_bound}
\end{align}
where $C, C', C'' > 0$ are some constants depending only on $d$. As $\lambda $ is bounded away from $0$, we may choose $c_d$ large enough such that the sum of the last two terms in \eqref{eq:induction_bound} is less than $0$ for any $r \geq 1$. Using \eqref{eq:P_Q_sum}, and by possibly increasing the constant $c_d$, one can extend this to the required upper bound on $\P_{\alpha, \kappa_c}(0 \leftrightarrow x_r)$. The lower bound on $\P_{\alpha, \kappa_c}(0 \leftrightarrow x_r)$ follows in a similar way.
\end{proof}

\smallskip

We proceed with the proof of the lower bound in \eqref{eq:beta-crit}. 
As in the proof of the lower bound for $\kappa>\kappa_c$,
with Lemma~\ref{L:two_point_critical} replacing the role of Theorem~\ref{T:two_point_order} (so that $\psi$ is replaced with a constant multiple of $\alpha$), we obtain that 
\[
\E_{\alpha, \kappa_c}[|N_r|] \geq \frac{\alpha}{c_d}  (d-1)^{r} \lambda^r \quad\text{and}\quad 
\E_{\alpha, \kappa_c}[|N_r|^2] \leq C(d-1)^{2r} \lambda^{2r},
\]
for $r$ large enough and 
\[
C = C(\alpha) = c_d \alpha^4 \frac{1}{1-\frac{1}{(d-1)\lambda}} \frac{1}{(1-\gamma)^3},
\]
where $1- \gamma = 1 - \tau_{\kappa_c}^2 \lambda^{-1} = 1 - \frac{1}{(d-1)\lambda}$. By Lemma~\ref{l:lambda-kappa-critical}, we obtain by enlarging $c_d$ that 
\[
\E_{\alpha, \kappa_c}[|N_r|^2] \leq c_d(d-1)^{2r} \lambda^{2r}.
\]
Hence, by the Paley-Zygmund inequality,
\[
\theta(\alpha,\kappa_c) = \lim\limits_{r\to\infty}\P_{\alpha,\kappa_c}\big(|N_r|\geq 1\big) \geq \frac{\E_{\alpha,\kappa_c}\big[|N_r|\big]^2}{\E_{\alpha,\kappa_c}\big[|N_r|^2\big]}
\geq c_d^{-3}\alpha^2,
\]
which completes the proof of the lower bound in \eqref{eq:beta-crit}.

\smallskip

We proceed with the proof of the upper bound in \eqref{eq:beta-crit}. The application of Lemma~\ref{l:Lyons} does not give the right decay in the case $\kappa=\kappa_c$: even though the effective conductance $C_{\text{eff}}(0 \leftrightarrow \infty)$ is of order $\alpha^2$, we do not have a uniform control in $\alpha$ of the constant $M$. Therefore, we use a different approach here. 
Recall the definition of $A_r$ from \eqref{eq:def_A_r} and note that for $\kappa=\kappa_c$, 
\[
A_r = 2 + \frac{d-2}{d-1}(r-1) \asymp r.
\]

By Lemma~\ref{L:single_loop_to_boundary}, the measure of loops intersecting $0$ and the boundary $\partial B(r)$ satisfies
\[
\mu_{\kappa_c}(\ell : 0 \in \ell, \ell \cap \partial B(r) \neq \emptyset) = - \log( 1 - \frac{\tau^{2r}_{\kappa_c} d(d-1)^{r-1}}{A_r}).
\]
In particular, using that $\tau^2_{\kappa_c} (d-1) = 1$, there exists a constant $c_d$ such that for $r$ large enough
\begin{equation}
\mu_{\kappa_c}(\ell : 0 \in \ell, \ell \cap \partial B(r) \neq \emptyset) \leq - \log(1 - c_d r^{-1}). \label{eq:bound_0_B(r)_single}
\end{equation}
For $L \in \N$, denote by $\cC_{\leq L}(0)$ the open cluster of $0$ using only loops of diameter at most $L$. Further, let $\cC_{\leq L}^1(0) \subseteq \cC_{\leq L}(0)$ be those vertices (other than $0$) that are connected to $0$ by a single loop of diameter at most $L$. Lemma~\ref{L:single_loop} gives
\begin{align*}
    \E_{\alpha, \kappa_c} \big[ | \cC_{\leq L}^1(0)| \big] &= \sum_{k=1}^L \sum_{v \in \partial B(k)} \P(\exists \ell \in \mathcal{L}_{\alpha, \kappa_c}\text{ with }0,v \in \ell, \, \text{diam}(\ell) \leq L) \\
    &\leq \sum_{k=1}^L d(d-1)^{k-1} \P(\exists \ell \in \mathcal{L}_{\alpha, \kappa_c} \, : \, 0, v \in \ell) \\
    &\leq 2 d \alpha \sum_{k=1}^{L}(d-1)^{k-1} (d-1)^{-k} =  2 \alpha L \frac{d}{d-1}.
\end{align*}
As similarly observed in \cite[Section~5]{CS16}, the size of $\cC_{\leq L}(0)$ is stochastically dominated by a Galton-Watson process with the offspring distribution given by the size of $\cC_{\leq L}^1(0)$. 
For
\[
L = \lfloor (d-1)/4d\alpha \rfloor,
\]
the Galton-Watson process is subcritical with mean offspring at most $1/2$. This implies that $\cC_{\leq L}(0)$ is almost surely finite with $\E[|\cC_{\leq L}(0)|]\leq 2$.

If $0$ is in an infinite cluster but $\cC_{\leq L}(0)$ is finite, there must be at least one vertex in $\cC_{\leq L}(0)$ contained in a loop of diameter at least $L+1$. By conditioning on $\mathcal{F}_{\leq L}$, the sigma-algebra generated by loops of diameter at most $L$, we see from \eqref{eq:bound_0_B(r)_single} that
\begin{align*}
    \P_{\alpha, \kappa_c}( |\mathcal C(0)|=\infty \mid \mathcal{F}_{\leq L}) &\leq \sum_{v \in \cC_{\leq L}(0)} \P( \exists \ell \in \mathcal{L}_{\alpha, \kappa_c}  \text{ with } v \in \ell, \, \text{diam}(\ell) > L) \\
    &\leq |\cC_{\leq L}(0)| \P( \exists \ell \in \mathcal{L}_{\alpha, \kappa_c}  \text{ with } 0 \in \ell, \, \ell\cap\partial B(\lfloor L/2\rfloor)\neq \emptyset)\\
    &\leq  |\cC_{\leq L}(0)| (1 - (1- c_d'\alpha)^\alpha),
\end{align*}
where $c_d'$ is some constant depending only on $d$. Using the bound $1-(1-x)^{\alpha} \leq 2x\alpha$ (if $\alpha$ is small enough), and then taking expectations, yields 
\[
\theta(\alpha, \kappa_c) \leq 4 c_d' \alpha^2,
\]
as required.\qed

\bibliographystyle{alpha}
\bibliography{loop_bib} 

\end{document}